%% file: smooth_structures_on_1_manif.tex
\documentclass[12pt, reqno]{amsproc}
\usepackage[utf8]{inputenc}
\usepackage[T2A]{fontenc}
\usepackage[english]{babel}
\usepackage{amsmath}
\usepackage{amsfonts}
\usepackage{amssymb}
\usepackage{amsthm}
\usepackage{enumitem}

\usepackage[svgcolors, dvipsnames]{xcolor}
\usepackage{hyperref}
\hypersetup{
    hypertexnames=false,
    colorlinks,
    linkcolor={red!50!black},
    citecolor={blue!50!black},
    urlcolor={blue!80!black}
}

\usepackage{graphicx}
\usepackage[all]{xy}
\usepackage[margin = 2.5cm]{geometry}

\input{macro_main.tex}

\input{macro_spec.tex}

\title[Smooth structures on the line with two origins]{Classification of differentiable structures on the non-Hausdorff line with two origins}

\author{Mykola Lysynskyi}
\address{Department of Algebra and Topology, Institute of Mathematics of NAS of Ukraine, Teresh\-chenkivska str. 3, Kyiv, 01601, Ukraine}
\email{m.lysynskyi@imath.kiev.ua}

\author{Sergiy Maksymenko}
\address{Department of Algebra and Topology, Institute of Mathematics of NAS of Ukraine, Teresh\-chenkivska str. 3, Kyiv, 01601, Ukraine}
\email{maks@imath.kiev.ua}


\begin{document}

\begin{abstract}
We classify differentiable structures on a line $\mathbb{L}$ with two origins being a non-Hausdorff but $T_1$ one-dimensional manifold obtained by ``doubling'' $0$.
For $k\in\mathbb{N}\cup\{\infty\}$ let $H$ be the group of homeomorphisms $h$ of $\mathbb{R}$ such that $h(0)=0$ and the restriction of $h$ to $\mathbb{R}\setminus0$ is a $\mathcal{C}^{k}$-diffeomorphism.
Let also $D$ be the subgroup of $H$ consisting of $\mathcal{C}^{k}$-diffeomorphisms of $\mathbb{R}$ also fixing $0$.
It is shown that there is a natural bijection between $\mathcal{C}^{k}$-structures on $\mathbb{L}$ (up to a $\mathcal{C}^{k}$-diffeomorphism fixing both origins) and double $D$-coset classes $D \setminus H / D = \{ D h D \mid h \in H\}$.

Moreover, the set of all $\mathcal{C}^{k}$-structures on $\mathbb{L}$ (up to a $\mathcal{C}^{k}$-diffeomorphism which may also exchange origins) are in one-to-one correspondence with the set of double $(D,\pm)$-coset classes $D \setminus H^{\pm} / D = \{ D h D \cup D h^{-1} D \mid h \in H\}$.

In particular, in contrast with the real line, the line with two origins $\mathbb{L}$ admits uncountably many pair-wise non-diffeomorphic $\mathcal{C}^{k}$-structures for each $k=1,2,\ldots,\infty$.
\end{abstract}
    
\maketitle

\input{text.tex}


\input{smooth_structures_on_1_manif_biblio.tex}
\end{document}

%% file: macro_main.tex
\newcommand\mycolor[1]{}

\setlist[enumerate]{itemsep=0.3ex, topsep=0.3ex, label={\rm(\arabic*)}}
\setlist[itemize]{itemsep=0.3ex, topsep=0.3ex, leftmargin=4ex}

\newtheorem{theorem}[subsection]{Theorem}

\newtheorem{corollary}[subsection]{Corollary}

\newtheorem{definition}[subsection]{Definition}

\newtheorem{subtheorem}[subsubsection]{Theorem}
\newtheorem{sublemma}[subsubsection]{Lemma}

\newtheorem{subcorollary}[subsubsection]{Corollary}

\newtheorem{subremark}[subsubsection]{Remark}
\newtheorem{subexample}[subsubsection]{Example}

\newtheorem{subproblem}[subsection]{Problem}

\makeatletter
\@addtoreset{subsection}{section}
\@addtoreset{equation}{section}
\@addtoreset{figure}{section}
\@addtoreset{table}{section}
\makeatother

\makeatletter
\newcommand\testshape{family=\f@family; series=\f@series; shape=\f@shape.}
\def\myemphInternal#1{\if n\f@shape%
\begingroup\itshape #1\endgroup\/%
\else\begingroup\sf\itshape\small #1\endgroup%
\fi}
\def\myemph{\futurelet\testchar\MaybeOptArgmyemph}
\def\MaybeOptArgmyemph{\ifx[\testchar \let\next\OptArgmyemph
                 \else \let\next\NoOptArgmyemph \fi \next}
\def\OptArgmyemph[#1]#2{\index{#1}\myemphInternal{#2}}
\def\NoOptArgmyemph#1{\myemphInternal{#1}}
\makeatother

\newcommand\term[2][\empty]{\myemph[#1]{#2}}

\newcommand\monoArrow{\lhook\joinrel\rightarrow}

\newcommand\epiArrow{\rightarrow\!\!\!\!\!\to}

\newcommand\Aman{A}

\newcommand\Mman{M}
\newcommand\Nman{N}

\newcommand\Uman{U}
\newcommand\Vman{V}
\newcommand\Wman{W}
\newcommand\Xman{X}

\newcommand\Agrp{A}

\newcommand\Cgrp{C}
\newcommand\Dgrp{D}

\newcommand\Hgrp{H}

\newcommand\Kgrp{K}

\newcommand\bD{\mathbb{D}}
\newcommand\bN{\mathbb{N}}

\newcommand\bR{\mathbb{R}}
\newcommand\bZ{\mathbb{Z}}

\newcommand\id{\mathrm{id}}          
\newcommand\Int{\mathrm{Int}}        
\newcommand\Cl[1]{\overline{#1}}     
\newcommand\eps{\varepsilon}                   

\newcommand\restr[2]{#1\vert_{#2}}
\newcommand\GL{\mathrm{GL}}

\newcommand\Aut{\mathrm{Aut}}       
\newcommand\Diff{\mathcal{D}}       
\newcommand\Homeo{\mathcal{H}}      
 
\newcommand\Cr[1]{\mathcal{C}^{#1}}
\newcommand\Cinfty{\mathcal{C}^{\infty}}

\newcommand\fixsymbol{\mathrm{fix}}
\newcommand\invsymbol{}
\newcommand\nbsymbol{\mathrm{nb}}
\newcommand\folsymbol{*}

\newcommand\DiffInv[3][\empty]{\Diff_{\invsymbol}(#2,#3\ifx\empty #1\relax\else,#1\fi)}

\newcommand\DiffFix[3][\empty]{\Diff_{\fixsymbol}(#2,#3\ifx\empty #1\relax\else,#1\fi)}

\newcommand\DiffNb[3][\empty]{\Diff_{\nbsymbol}(#2,#3\ifx\empty #1\relax\else,#1\fi)}

\newcommand\DiffHFix[3][\empty]{\Diff^{0}_{\fixsymbol}(#2,#3\ifx\empty #1\relax\else,#1\fi)}

\newcommand\DiffHNb[3][\empty]{\Diff^{0}_{\nbsymbol}(#2,#3\ifx\empty #1\relax\else,#1\fi)}

\newcommand\DiffPlusFix[3][\empty]{\Diff^{+}_{\fixsymbol}(#2,#3\ifx\empty #1\relax\else,#1\fi)}

\newcommand\FDiff[2][\empty]{\Diff^{\folsymbol}(#2\ifx\empty #1\relax\else,#1\fi)}
\newcommand\FDiffFix[2][\empty]{\Diff^{\folsymbol}_{\fixsymbol}(#2\ifx\empty #1\relax\else,#1\fi)}
\newcommand\FDiffA[2][\empty]{\Diff^{=}(#2\ifx\empty #1\relax\else,#1\fi)}

\newcommand\VBAut[2][\empty]{\GL(#2\ifx\empty #1\relax\else,#1\fi)}

\newcommand\DiffLP{\Diff}  
\newcommand\DiffLPInv[3][\empty]{\DiffLP_{inv}(#2,#3\ifx\empty#1\relax\else,#1\fi)}

\newcommand\DiffLPFix[3][\empty]{\DiffLP_{fix}(#2,#3\ifx\empty#1\relax\else,#1\fi)}

\newcommand\DiffLPNb[3][\empty]{\DiffLP_{nb}(#2,#3\ifx\empty#1\relax\else,#1\fi)}

\newcommand\func{f}
\newcommand\gfunc{g}
\newcommand\dif{h}
\newcommand\gdif{g}
\newcommand\kdif{k}
\newcommand\pdif{p}
\newcommand\qdif{q}

\newcommand\px{x}
\newcommand\py{y}



%% file: macro_spec.tex
\newcommand\kk{k}
\newcommand\Ck{\Cr{\kk}}

\newcommand\CrAB[5]{\mathcal{C}^{#1}_{#2,#3}(#4,#5)}
\newcommand\DrAB[5]{\Diff^{#1}_{#2,#3}(#4,#5)}
\newcommand\DrA[3]{\Diff^{#1}_{#2}(#3)}

\newcommand\CkABMN{\CrAB{\kk}{\Ustruct}{\Vstruct}{\Mman}{\Nman}}
\newcommand\DiffkABMN{\DrAB{\kk}{\Ustruct}{\Vstruct}{\Mman}{\Nman}}
\newcommand\DiffkAM{\DrA{\kk}{\Ustruct}{\Mman}}

\newcommand\attt[1]{\bar{#1}}

\newcommand\Uset{U}
\newcommand\Vset{V}
\newcommand\Wset{W}
\newcommand\Yset{Y}

\newcommand\Umap{\varphi}
\newcommand\Vmap{\psi}
\newcommand\Wmap{\sigma}
\newcommand\Ymap{\xi}

\newcommand\Uchr{\mathsf{u}}
\newcommand\Vchr{\mathsf{v}}
\newcommand\Wchr{\mathsf{w}}
\newcommand\Ychr{\mathsf{y}}

\newcommand\tUmap{\attt{\Umap}}
\newcommand\tVmap{\attt{\Vmap}}

\newcommand\tUchr{\attt{\Uchr}}

\newcommand\ui{i}
\newcommand\vi{j}

\newcommand\Uchrii[1]{\Uchr_{#1}}
\newcommand\Vchrii[1]{\Vchr_{#1}}
\newcommand\Wchrii[1]{\Wchr_{#1}}

\newcommand\Usetii[1]{\Uset_{#1}}
\newcommand\Vsetii[1]{\Vset_{#1}}
\newcommand\Wsetii[1]{\Wset_{#1}}

\newcommand\Umapii[1]{\Umap_{#1}}
\newcommand\Vmapii[1]{\Vmap_{#1}}
\newcommand\Wmapii[1]{\Wmap_{#1}}

\newcommand\Vchri{\Vchrii{\vi}}

\newcommand\Useti{\Usetii{\ui}}
\newcommand\Vseti{\Vsetii{\vi}}

\newcommand\Umapi{\Umapii{\ui}}
\newcommand\Vmapi{\Vmapii{\vi}}

\newcommand\UInd{\Lambda}

\newcommand\UAtlas{\mathsf{A}}
\newcommand\VAtlas{\mathsf{B}}

\newcommand\Ustruct{\mathfrak{A}}
\newcommand\Vstruct{\mathfrak{B}}

\newcommand\UVTrMap{\gdif}

\newcommand\aii[1]{a_{#1}}
\newcommand\bii[1]{b_{#1}}
\newcommand\cii[1]{c_{#1}}
\newcommand\dii[1]{d_{#1}}

\newcommand\CanonAtlas[1]{\mathsf{Id}_{#1}}
\newcommand\CanonStr[1]{\mathcal{I}_{#1}}

\newcommand\ua{0}
\newcommand\ub{1}
\newcommand\uc{2}
\newcommand\ud{3}

\newcommand\hcl[1]{\mathsf{hcl}(#1)} 

\newcommand\DLine{\mathbb{L}}    
\newcommand\topX{\eta} 

\newcommand\pta{0}               
\newcommand\ptb{\bar{\pta}}      

\newcommand\Na{\Uset}
\newcommand\Nb{\Vset}

\newcommand\DiffLx[2][\empty]{\Diff^{#2}_{#1}(\DLine)}

\newcommand\DiffLxOrPres[1][\empty]{\DiffLx[#1]{+}}

\newcommand\DiffLxOFix[1][\empty]{\DiffLx[#1]{\mathrm{fix}}}
\newcommand\DiffLxOEx[1][\empty]{\DiffLx[#1]{\mathrm{ex}}}

\newcommand\DiffLxOrPresxOFix[1][\empty]{\DiffLx[#1]{+,\mathrm{fix}}}
\newcommand\DiffLxOrPresxOEx[1][\empty]{\DiffLx[#1]{+,\mathrm{ex}}}
\newcommand\DiffLxOrRevxOFix[1][\empty]{\DiffLx[#1]{-,\mathrm{fix}}}
\newcommand\DiffLxOrRevxOEx[1][\empty]{\DiffLx[#1]{-,\mathrm{ex}}}

\newcommand\perm{\mathsf{p}}

\newcommand\idU{\id_{u}}
\newcommand\idV{\id_{v}}

\newcommand\Rzp{\bR\setminus0}

\newcommand\dbl[3]{#1\setminus #2 \,/\, #3}

\newcommand\ael{a}
\newcommand\bel{b}
\newcommand\cel{c}
\newcommand\del{d}
\newcommand\gel{g}
\newcommand\hel{h}
\newcommand\kel{k}
\newcommand\lel{l}
\newcommand\gCHD{\dbl{\Cgrp}{\Hgrp}{\Dgrp}}

\newcommand\dbli[2]{#1\setminus #2^{\pm1} \,/\, #1}

\newcommand\trmap[1]{\gdif_{#1}}
\newcommand\Utrmap{\trmap{\UAtlas}}
\newcommand\Vtrmap{\trmap{\VAtlas}}

\newcommand\CkStructs[1]{\mathcal{C}^{\kk}(#1)}
\newcommand\IsoClassesLFix{\CkStructs{\DLine}/\HomeoLxOFix}
\newcommand\IsoClassesL{\CkStructs{\DLine}/\HomeoL}

\newcommand\trStruct[1]{\mathfrak{S}(#1)}
\newcommand\MinAtlas[1]{\mathsf{S}(#1)}

\newcommand\adif{a}
\newcommand\bdif{b}

\newcommand\qfunc{q}

\newcommand\stlin[1]{w_{#1}}

\newcommand\aHRzp{\mathcal{W}^{\kk}(\bR,0)}

\newcommand\aDRnbp{\Diff^{\kk}_{\mathrm{nb}}(\bR,0)}
\newcommand\DRzp{\Diff^{\kk}_{\mathrm{mon}}(\bR\setminus0)}

\newcommand\invol{\zeta}
\newcommand\grpInvol{\langle\invol\rangle}
\newcommand\BranchPtDLine{\{\pta,\ptb\}}
\newcommand\rmap{\mathsf{r}_{\Na}}

\newcommand\bijDL[2][\empty]{\xi^{\mathrm{#2}}_{#1}}
\newcommand\bijDLxOFix[1][\empty]{\bijDL[#1]{\mathrm{fix}}}
\newcommand\bijDLxOEx[1][\empty]{\bijDL[#1]{\mathrm{ex}}}

\newcommand\bijDfix{\bijDLxOFix[\UAtlas,\VAtlas]}
\newcommand\bijDex{\bijDLxOEx[\UAtlas,\VAtlas]}

\newcommand\inv[1]{1/#1} 

\newcommand\HRz[1][\kk]{\mathcal{W}^{#1}(\bR,0)}
\newcommand\HRzOrPres{\HRz[\kk,+]}

\newcommand\DiffRz[1][\kk]{\Diff^{#1}(\bR,0)}
\newcommand\DiffRzOrPres{\DiffRz[\kk,+]}
\newcommand\DiffRzOrRev{\DiffRz[\kk,-]}

\newcommand\JDiffRz[1][\kk]{\mathcal{E}^{#1}(\bR,0)}
\newcommand\JDiffRzOrPres[1][\kk]{\JDiffRz[\kk,+]}
\newcommand\JDiffRzOrRev[1][\kk]{\JDiffRz[\kk,-]}

\newcommand\HomeoL{\Homeo(\DLine)}
\newcommand\HomeoLx[2][\empty]{\Homeo^{#2}_{#1}(\DLine)}
\newcommand\HomeoLxOrPres[1][\empty]{\HomeoLx[#1]{+}}

\newcommand\HomeoLxOFix[1][\empty]{\HomeoLx[#1]{\mathrm{fix}}}

\newcommand\HomeoLxOrPresxOFix[1][\empty]{\HomeoLx[#1]{+,\mathrm{fix}}}

\newcommand\HomeoRz[1][\empty]{\Homeo^{#1}(\bR,0)}
\newcommand\HomeoRzOrPres{\HomeoRz[+]}
\newcommand\HomeoRzOrRev{\HomeoRz[-]}

\newcommand\ax{a}
\newcommand\bx{b}

%% file: text.tex

\section{Introduction}\label{sect:introduction}
The present paper is devoted to some steps toward study and classification of smooth structures on non-Hausdorff manifolds of dimension $1$.

It is well known that each Hausdorff manifold $\Mman$ of dimension $\dim\Mman\leq 3$ has a unique smooth structure up to a diffeomorphism.
For $\dim\Mman=1$ this is classical result, and for $\dim\Mman=2,3$ see Munkres~\cite[Theorem~6.3]{Munkres:AnnMath:1968} and Whitehead~\cite[Corollary~1.18]{Whitehead:AnnMath:1961}.
Also, Stallings~\cite{Stallings:PCPS:1962} shown that $\bR^{n}$ has a unique $\Cinfty$ structure for $n\geq 5$, see also~\cite[Section~10]{Ferry:GeomTop}.
On the other hand, $\bR^{4}$ has uncountably many pair-wise non-diffeomorphic $\Cinfty$ structures, see Freedman~\cite{Freedman:JDG:1982}, Donaldson~\cite[Theorem~1]{Donaldson:JDG:1983}, Freed and Uhlenbeck~\cite{FreedUhlenbeck:Inst:1991}, Gompf~\cite{Gompf:JDG:1985} and Taubes~\cite{Taubes:JDG:1987}.
It is also worth noting a result by Milnor~\cite{Milnor:AM:1956} on existence of distinct smooth structures on $7$-sphere, and also papers by Kervaire \& Milnor~\cite{KervaireMilnor:AM:1963} on the group of homotopy spheres, and Kervaire~\cite{Kervaire:CMH:1960} on $10$-manifold which does not admit any differentible structures.

On the other hand, differentiable structures of non-Hausdorff manifolds seem to be not essentially and systematically studied in the literature, see e.g.\ the book by D.~Gauld~\cite{Gauld:NonMetrManif:2014} and references therein.
For instance, Haefliger and Reeb~\cite{HaefligerReeb:EM:1957} proven existence of smooth structures on \term{non-Hausdorff letter $Y$}.
Nyikos~\cite{Nyikos:AM:1992} studied smooth structures on the \term{long line} $X$, and shown that it admits non-diffeomorphic $\Cr{\omega}$ (real analytic) structures, however it is not clear whether there are non-diffeomorphic $\Cr{\infty}$-structures on $X$.
The long line and the \term{Pr\"{u}fer manifold} are also considered in~\cite[Appendix~A]{Spivak:DG:1:1979}.
Applications of non-Hausdorff manifold to physics are discussed in~\cite{Hajicek:CMP:1971,HellerPysiakSasin:JMP:2011}.
Also, in a recent papers~\cite{OConnel:TA:2023, OConnel:TA:2024} by D.~O'Connel the structure of vector bundles over non-Hausdorff manifolds are also investigated.

In the present paper we will completely classify $\Ck$-structures, $1\leq\kk\leq\infty$, on another well-known non-Hausdorff one-dimensional manifold, called \term{line with two origins} $\DLine$, and which can be described in (at least) the following three equivalent ways, see Figure~\ref{fig:dline}:
\begin{enumerate}[leftmargin=*, itemsep=1ex]
\item
Let $R = (\bR\times 0) \sqcup (\bR\times 1)$ be a disjoint union of two copies of $\bR$.
Then $\DLine$ is a quotient space of $R$ obtained by identifying points $(t,0)$ with $(t,1)$ for $t \ne 0$.

\item
Let $\tau$ be the standard topology on $\bR$.
Then $\DLine = \bR\sqcup\ptb$ is a disjoint union of $\bR$ with some point $\ptb$ endowed with the following topology:
\[
    \topX = \tau \cup \{  (\Wset\setminus\pta)\cup\{\ptb\} \mid \pta\in\Wset\in\tau  \},
\]
whose elements are elements of $\tau$ and also open neighborhoods of $\pta$ in which $\pta$ is replaced with $\ptb$.

\item
Let $\mathcal{F}$ be the partition of $\bR^2\setminus0$ into connected components of the intersections of $\bR^2\setminus0$ with vertical lines $\{x=c\}_{c\in\bR}$.
Then $\DLine$ is the space of leaves $(\bR^2\setminus0)/\mathcal{F}$ of that foliation (with the corresponding quotient topology), so that the leaves corresponding to non-distinguishable points $\pta$ and $\ptb$ are the open intervals $0\times(-\infty;0)$ and $0\times(0;+\infty)$.
\end{enumerate}

\begin{figure}[!htbp]
\includegraphics[height=4cm]{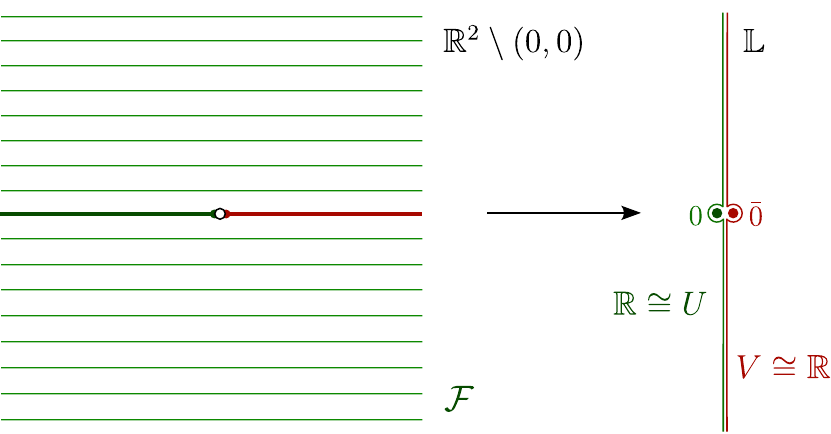}
\caption{Line with two origins $\DLine$}\label{fig:dline}
\end{figure}

It will be convenient to exploit the second model for $\DLine$.
As in Figure~\ref{fig:dline} denote
\begin{align}\label{equ:sets_U_V}
    &\Na := \bR = \DLine\setminus\ptb, &
    &\Nb := \DLine\setminus\pta = (\bR\setminus\pta)\cup\ptb.
\end{align}
\begin{sublemma}
$\DLine$ is locally Euclidean, one-dimensional, second countable, satisfies $T_1$ separation axiom, but does not satisfy $T_2$.
\end{sublemma}
\begin{proof}
This lemma is trivial and we will only explain the non-$T_2$ property.
Note that $\Na$ is an open neighborhood of $\pta$ which does not contain $\ptb$, while $\Nb$ is an open neighborhood of $\pta$ which does not contain $\pta$.
On the other hand, by the definition of topology $\topX$, any pair of neighborhoods of $\pta$ and $\ptb$ intersects.
Thus, $\DLine$ is $T_1$ but not $T_2$.
\end{proof}

Our main result is an explicit description of
\begin{itemize}[itemsep=1ex, topsep=1ex]
\item the set of differentiable $\Ck$-structures on $\DLine$ up to a $\Ck$-diffeomorphism and
\item the sets of $\Ck$-diffeomorphisms between $\Ck$-structures, see Theorem~\ref{th:Ck-struct_on_L_simple}.
\end{itemize}
As a consequence we will show that, in contrast with the real line $\bR$, on $\DLine$ there exists uncountably many pair-wise non-diffeomorphic structures, see Theorem~\ref{th:ck_structs_examples}.
The exposition is intended to be elementary and self-contained and is aimed to be available for a large audience of readers interesting in these topics.

\subsection{Notation}\label{sect:notation}
For each $\kk\in\bN\cup\{\infty\}$ define the following groups:
\begin{itemize}[leftmargin=*, itemsep=1ex]
\item $\HomeoRz$ is the group of homeomorphisms of $\bR$ fixing $\pta$;
\item $\HomeoRzOrPres$ be its normal subgroup (of index $2$) consisting of orientation preserving homeomorphisms, and
$\HomeoRzOrRev$ be the subset consisting of homeomorphisms reversing orientation;
\item $\HRz$ is the group of homeomorphisms $\dif$ of $\bR$ such that $\dif(0)=0$ and the restriction $\restr{\dif}{\Rzp}\colon\Rzp\to\Rzp$ is a $\Ck$-diffeomorphism of $\Rzp$;
\item $\DiffRz$ is the subgroup of $\aHRzp$ consisting of $\Ck$-diffeomorphisms of $\bR$;
\item $\JDiffRz$ is the subgroup of $\DiffRz$ consisting of diffeomorphisms $\dif$ such that $\dif^{(i)}(0)=0$ for $i=2,3,\ldots,\kk$;
\item $\aDRnbp$ is the group of $\Ck$-diffeomorphisms fixed on some neighborhood of $0$;
\item also for $* \in \{\pm\}$ and $\mathcal{X} \in \{\mathcal{W}, \Diff,\mathcal{E}\} $ put $\mathcal{X}^{\kk,*}(\bR,0):=\HomeoRz[*]\cap\mathcal{X}^{\kk}(\bR,0)$.
\end{itemize}

Notice that for all $\kk'>\kk$ we have the following inclusions
\[
     \xymatrix@R=1.2em{
        \JDiffRz \ar@{^(->}[r] \ar@{^(->}[d]  & \DiffRz  \ar@{^(->}[r] \ar@{^(->}[d] & \HRz \ar@{^(->}[d] \\
        \JDiffRz[\kk'] \ar@{^(->}[r] & \DiffRz[\kk']  \ar@{^(->}[r]  & \HRz[\kk']
     }
\]
Also, if $\kk$ is finite, then every $\dif\in\JDiffRz$ is given by the formula $\dif(\px) = \tau\px + \px^{\kk}\gdif(\px)$ for some $\tau \ne 0$ and a continuous function $\gdif\colon\bR\to\bR$ being $\Ck$ on $\Rzp$.
On the other hand, if $\kk=\infty$, then $\dif(\px) = \tau\px + \gdif(\px)$ for some $\tau\ne0$ and a $\Cinfty$ function being \term{flat} at $0$, i.e.\ $\gdif^{(i)}(0)=0$ for all $i\in\bN$.
In particular, $\JDiffRz[1] = \DiffRz[1]$.

Finally note that $\aDRnbp \subset \JDiffRzOrPres$.

\begin{subexample}\label{exmp:h_in_Hrzp}\rm
For any real numbers $\alpha,\beta  \ne 0$ having the same sign, and $s,t >0$ the following homeomorphism
\[
    \dif\colon\bR\to\bR,
    \qquad
    \dif(\px) =
    \begin{cases}
        \alpha (-\px)^s, & \px<0, \\
        \beta \px^t, & \px \geq 0.
    \end{cases}
\]
belongs to $\aHRzp$.
Evidently, $\dif\in\DiffRz$ iff $\alpha=\beta$ and $t=s=1$, i.e.\ when $\dif$ is linear, and in that case it belongs even to $\JDiffRz[\infty]$.
\end{subexample}

Consider the following \term{left} action of the product $\DiffRz\times\DiffRz$ on $\aHRzp$ defined by
\[
    \mu\colon \DiffRz\times\DiffRz \times \aHRzp \to \aHRzp,
    \qquad
    \mu(a,b,\dif) = a \circ \dif \circ b^{-1}.
\]
The corresponding orbit space of $\aHRzp$ is called the \term{set of $\DiffRz$-double cosets} and denoted by
\[
    \dbl{\DiffRz}{\aHRzp}{\DiffRz}.
\]
One easily checks, see Lemma~\ref{lm:CHD_partition} below, that $\gdif,\dif\in\aHRzp$ belong to the same orbit of the above action iff $\DiffRz \, \cap \, \dif \DiffRz \gdif^{-1}  \ne \varnothing$.

Also, consider a partition of $\aHRzp$ into larger sets.
Say that $\gdif,\dif\in\aHRzp$ are $(\DiffRz,\pm)$-equivalent, if \term{either $\gdif$ or $\gdif^{-1}$ belong to the orbit of $\dif$ with respect to the above action $\mu$}.
One easily checks that this relation is an equivalence, see Section~\ref{sect:pm_double_cosets}.
The corresponding equivalence classes will be called \term{$(\DiffRz,\pm)$-double cosets}, and the set of all classes will be denoted by
\[
    \dbli{\DiffRz}{\aHRzp}.
\]
It can also be defined as the orbit set of a certain natural action of the \term{wreath product} $\DiffRz\wr\bZ_{2}$ exstending the action $\mu$, see Section~\ref{sect:pm_double_cosets}.

\subsection{Main results}
Let $\HomeoL$ be the group of all homeomorphisms of $\DLine$.
It is easy to show that every homemorphism $\dif$ of $\DLine$ either fixes or exchanges points $\pta$ and $\ptb$, see Lemma~\ref{lm:branch_pt_invariant} below.
In particular, the restriction of $\dif$ to the complement $\DLine\setminus\{\pta,\ptb\} = \Rzp$ is also invariant under $\dif$.
Moreover, $\restr{\dif}{\Rzp}\colon\Rzp\to\Rzp$ is either strictly increasing or strictlty decreasing homeomorphism.
Therefore we will also say that $\dif$ \term{preserves} orientation of $\DLine$ if $\restr{\dif}{\Rzp}$ is  increasing, and $\dif$ \term{reverses} orientation of $\DLine$ otherwise.

For $*\in\{+,-,\mathrm{fix},\mathrm{ex}\}$ denote by $\HomeoLx{*}$ the \term{subset} of $\HomeoL$ consisting of homeomorphisms respectively \term{preserving} or \term{reversing} orientation of $\DLine$, or \term{fixing} or \term{exchanging} $\pta$ and $\ptb$.
Also for $*\in\{+,-\}$ and $\bullet\in\{\mathrm{fix},\mathrm{ex}\}$ we denote $\HomeoLx{*,\bullet}=\HomeoLx{*}\cap\HomeoLx{\bullet}$.
Evidently, $\HomeoLxOFix$ and $\HomeoLxOrPres$ are normal subgroups of $\HomeoL$ of index $2$, while their intersection $\HomeoLxOrPresxOFix$ has index $4$.

Fix once and for all some $\kk\in\bN\cup\{\infty\}$.
Given two $\Ck$-structures $\Ustruct$ and $\Vstruct$ on $\DLine$ let:
\begin{itemize}[itemsep=1ex, leftmargin=*]
\item
$\DiffLx[\Ustruct,\Vstruct]{}$ be the set of all $\Ck$-diffeomorphisms of $\DLine$ from $\Ustruct$ to $\Vstruct$;

\item
$\DiffLx[\Ustruct,\Vstruct]{*} = \DiffLx[\Ustruct,\Vstruct]{} \cap \HomeoLx{*}$, where $*\in\{-,+,\mathrm{fix},\mathrm{ex}\}$;

\item
$\DiffLx[\Ustruct,\Vstruct]{*,\bullet} = \DiffLx[\Ustruct,\Vstruct]{*} \cap \DiffLx[\Ustruct,\Vstruct]{\bullet} =
\DiffLx[\Ustruct,\Vstruct]{} \cap \HomeoLx{*,\bullet}$, where $*\in\{-,+\}$, and $\bullet \in \{\mathrm{fix}, \mathrm{ex}\}$.
\end{itemize}
Notice that though we omitted $\kk$ from notations, it is implicitly determined by the notations for $\Ck$-structures $\Ustruct$ and $\Vstruct$.

If $\Ustruct=\Vstruct$, then we will abbreviate $\DiffLx[\Ustruct,\Vstruct]{*}$ and $\DiffLx[\Ustruct,\Vstruct]{*,\bullet}$ to $\DiffLx[\Ustruct]{*}$ and $\DiffLx[\Ustruct]{*,\bullet}$ respectively.
For instance, $\DiffLxOrPres[\Ustruct]$, resp.\ $\DiffLxOFix[\Ustruct]$, is the \term{group} of $\Ck$-diffeomorphisms of $\Ustruct$ preserving orientation, resp.\ fixing $\pta$ and $\ptb$, while $\DiffLxOrPresxOFix[\Ustruct]$ is their intersection.

Evidently, we have the following partition
\begin{align*}
    \DiffLx[\Ustruct,\Vstruct]{}          \ = \
    \DiffLxOrPresxOFix[\Ustruct,\Vstruct] \ \sqcup \
    \DiffLxOrRevxOFix[\Ustruct,\Vstruct]  \ \sqcup \
    \DiffLxOrPresxOEx[\Ustruct,\Vstruct]  \ \sqcup \
    \DiffLxOrRevxOEx[\Ustruct,\Vstruct],
\end{align*}
however depending on $\Ustruct$ and $\Vstruct$ some of $\DiffLx[\Ustruct,\Vstruct]{*,\bullet}$ may be empty.

The following theorem classifies $\Ck$-structures on $\DLine$.
In fact, we will prove more detailed and explicit statements, describing also sets of diffeomorphisms between $\Ck$-structures fixing or exchanging $\pta$ and $\ptb$, see Theorem~\ref{th:Ck-struct_on_L_detailed}.
\begin{subtheorem}\label{th:Ck-struct_on_L_simple}
Let $\kk\in\bN\cup\{\infty\}$.
Then
\begin{itemize}[leftmargin=*, itemsep=1ex]
\item
$\Ck$-structures on $\DLine$ \term{up to a $\Ck$-diffeomorphism} are in one-to-one correspondence with the set $\dbli{\DiffRz}{\aHRzp}$ of $(\DiffRz,\pm)$-double coset classes;
\item
while $\Ck$-structures on $\DLine$ \term{up to a $\Ck$-diffeomorphism fixing $\pta$ and $\ptb$} are in one-to-one correspondence with the set $\dbl{\DiffRz}{\aHRzp}{\DiffRz}$ of $\DiffRz$-double coset classes.
\end{itemize}
\end{subtheorem}

The following lemma coincides with statements~\ref{enum:th:Dfix_Dex:self_bij} and~\ref{enum:th:Dfix_Dex:DAB_pm_fixex} of Theorem~\ref{th:Dfix_Dex} below.

\begin{sublemma}
For any $\Ck$-structure, $\DiffLxOrPresxOFix[\Ustruct]$ contains a normal subgroup isomorphic to the group $\aDRnbp$ of diffeomorphisms of $\bR$ fixed near $0$.
In particular, $\DiffLxOrPresxOFix[\Ustruct]$ is always uncountable.

More generally, let $\Ustruct$ and $\Vstruct$ be $\Ck$-structures on $\DLine$ and $*\in\{\mathrm{fix},\mathrm{ex}\}$.
If $\DiffLx[\Ustruct,\Vstruct]{*,\bullet}$ is non-empty, then it is uncountable as well.
\end{sublemma}

Our last statement presents examples of $\Ck$-structures from which either of sets $\DiffLxOFix[\Ustruct,\Vstruct]$ and $\DiffLxOEx[\Ustruct,\Vstruct]$ is empty or non-empty.

A $\Ck$-atlas $\UAtlas$ on $\DLine$ will be called \term{minimal} if it consisting of only two charts $\Umap\colon\Na\to\bR$ and $\Vmap\colon\Nb\to\bR$ being homeomorphisms onto and such that $\Umap(\pta)=\Vmap(\ptb)=0$.
We will show in Lemma~\ref{lm:minimal_atlas_exist} that every $\Ck$-structure on $\DLine$ admits a minimal atlas.

In particular, let $\idU\colon\Na\to\bR$ be the identity map, and $\idV\colon\Vman\to\bR$ be the homeomorphism given by $\idV=\idU$ on $\Vset\setminus\{\pta\}$ and $\idV(\ptb)=0$.
Then for each $\gdif\in\aHRzp$ one can define the following minimal atlas $\MinAtlas{\gdif} = \{ (\Na, \idU), (\Nb, \gdif\circ\idV) \}$.
Such an atlas will be called \term{special minimal of $\gdif$}, see Section~\ref{sect:spec_minimal_atlas}.
Evidently, its transition map is
\[
\restr{\gdif\circ\idV\circ\idU^{-1}}{\Rzp} = \restr{\gdif}{\Rzp},
\]
whence $\MinAtlas{\gdif}$ is $\Ck$.

Further, for each $\ax>0$ define the following homeomorphism $\stlin{\ax} \in \aHRzp$ by
\begin{equation}\label{equ:stlin_b}
    \stlin{\ax}(\px) =
    \begin{cases}
        \px,   & \px \leq 0, \\
        \ax\px, & \px>0,
    \end{cases}
\end{equation}
see Figure~\ref{fig:wa}.

\begin{figure}[!htbp]
\includegraphics[height=3cm]{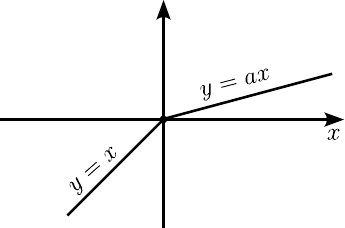}
\caption{Homeomorphism $\stlin{a}$}\label{fig:wa}
\end{figure}

It is a particular case of homeomorphism $\dif$ from Example~\ref{exmp:h_in_Hrzp} for $\alpha=1$, $\beta=\ax$, $s=t=1$.
Let $\trStruct{\ax}$ be the corresponding $\Ck$-structure defined by the minimal atlas $\MinAtlas{\stlin{\ax}}$.

\begin{subtheorem}\label{th:ck_structs_examples}
\begin{enumerate}[leftmargin=*, label={\rm(\arabic*)}, itemsep=1ex]
\item\label{enum:th:ck_structs_examples:a_1}
Suppose $\ax=1$, so $\stlin{\ax}=\id_{\bR}$.
Then there is an isomorphism
\[
    \eta\colon\DiffLx[\trStruct{1}]{} \cong \DiffRz \times \bZ_2,
\]
such that for $*\in\{\pm\}$
\begin{align}\label{equ:eta_a=1}
    &\eta\bigl( \DiffLx[\trStruct{1}]{*,\mathrm{fix}} \bigr) = \DiffRz[\kk,*]\times \{1\},&
    &\eta\bigl( \DiffLx[\trStruct{1}]{*,\mathrm{ex}} \bigr) = \DiffRz[\kk,*]\times \{-1\}.
\end{align}

\item\label{enum:th:ck_structs_examples:a_ne_1}
Let $\ax \ne 1$.
Then $\DiffLxOrPresxOEx[\trStruct{\ax}]{} = \DiffLxOrRevxOFix[\trStruct{\ax}]{} = \varnothing$, and $\DiffLx[\trStruct{\ax}]{}=\DiffLxOrPresxOFix[\trStruct{\ax}]{}\sqcup\DiffLxOrRevxOEx[\trStruct{\ax}]{}$, so every self-diffeomorphism of $\trStruct{\ax}$ either fixes $\pta$ and $\ptb$ and \term{preserves} orientation, or exchanges $\pta$ and $\ptb$ and \term{reverses} orientation.
Moreover, there exists an isomorphism
\[
    \eta\colon\DiffLxOrPresxOFix[\trStruct{\ax}]{} \cong \JDiffRzOrPres,
\]
and an element $\pdif \in \DiffLxOrRevxOEx[\trStruct{\ax}]{}$ of order $2$.
Hence, $\DiffLx[\trStruct{\ax}]{}$ splits into a semidirect product $\DiffLxOrPresxOFix[\trStruct{\ax}]{} \rtimes \langle\pdif\rangle$, and thus it is isomorphic to a certain%
\footnote{Notice that $\JDiffRzOrRev$ also contains the diffeomorphism $\qdif\colon\bR\to\bR$, $\qdif(\px)=-\px$, of order $2$, and therefore $\JDiffRz$ also splits into the semidirect product $\JDiffRzOrPres\rtimes\langle\qdif\rangle$.
However, \term{it is not clear} whether there is an isomorphism between $\DiffLx[\trStruct{\ax}]{}\cong \DiffLxOrPresxOFix[\trStruct{\ax}]{} \rtimes \langle\pdif\rangle$ and $\JDiffRz \cong \JDiffRzOrPres\rtimes\langle\qdif\rangle$ extending $\eta$.} semidirect product $\JDiffRzOrPres\rtimes\bZ_{2}$.

\item\label{enum:th:ck_structs_examples:a_ainv}
Also, for each $a \ne 1$ the structures $\trStruct{\ax}$ and $\trStruct{\inv{\ax}}$ are $\Ck$-diffeomorphic.
In fact, $\DiffLxOrPresxOFix[\trStruct{\ax},\trStruct{\inv{\ax}}] = \DiffLxOrRevxOEx[\trStruct{\ax},\trStruct{\inv{\ax}}] = \varnothing$ and
\[
    \DiffLx[\trStruct{\ax},\trStruct{\inv{\ax}}]{} =
        \DiffLxOrPresxOEx[\trStruct{\ax},\trStruct{\inv{\ax}}]
        \sqcup
        \DiffLxOrRevxOFix[\trStruct{\ax},\trStruct{\inv{\ax}}],
\]
so every $\Ck$-diffeomorphism from $\trStruct{\ax}$ to $\trStruct{\inv{\ax}}$ either fixes $\pta$ and $\ptb$ and \term{reverses} orientation, or exchanges $\pta$ and $\ptb$ and \term{preserves} orientation.
Moreover, for every $\dif\in\DiffLxOrRevxOFix[\trStruct{\ax},\trStruct{\inv{\ax}}]$ and $\gdif\in\DiffLxOrPresxOEx[\trStruct{\ax},\trStruct{\inv{\ax}}]$ we have the following bijections:
\[
    \DiffLxOrPresxOEx[\trStruct{\ax},\trStruct{\inv{\ax}}]
    \xrightarrow{~~\kdif\,\mapsto\,\gdif^{-1}\circ\kdif~~}
    \DiffLxOrPresxOFix[\trStruct{\ax}]
    \xrightarrow{~~\kdif\,\mapsto\,\dif\circ\kdif~~}
    \DiffLxOrRevxOFix[\trStruct{\ax},\trStruct{\inv{\ax}}].
\]

\item\label{enum:th:ck_structs_examples:a_ne_b_1b}
Finally, if $\ax \ne \bx,\inv{\bx}$, then $\DiffLx[\trStruct{\ax},\trStruct{\bx}]{}=\varnothing$, i.e.\ these $\Ck$-structures are not diffeomorphic.
In particular, $\DLine$ has uncountably many pair-wise non-diffeomorphic $\Ck$-structures of each $k\in\bN\cup\infty$.
\end{enumerate}
\end{subtheorem}

\begin{subremark}\rm
Notice that we have a surjective homomorphism $j^1\colon\JDiffRzOrPres \to \Rzp$, $j^1(\dif) = \dif'(0)$.
Let $G_{\kk}$ be its kernel.
Then $G_{\kk}$ consists of diffeomorphisms $\kk$-tangent to the identity and was studied in~\cite{Sternberg:DMJ:1957}.
Moreover, the group $G_{\infty}$ plays an important role in the theory of foliations, see~\cite{Takens:AIF:1973, Sergeraert:IM:1977, Kawabe:ASPM:1985}.
\end{subremark}

\subsection*{Structure of the paper}
Section~\ref{sect:preliminaries} is preliminary: it contains definitions of double cosets and branch points of non-Hausdorff spaces.
In Section~\ref{sect:Ck_structs_on_loc_eu_spaces} we discuss differentiable structures on locally Euclidean spaces not assuming them to be manifolds in a usual sense, i.e.\ Hausdorff and second countable.
In Section~\ref{sect:diff_struct_on_bR} we present a proof of a well-known result that any two $\Ck$-structures, $1\leq\kk \leq \infty$, on the real line are diffeomorphic, see Corollary~\ref{cor:Ck_struct_on_R}.
In fact we need a more explicit statement claiming that if $\Mman$ is a one-dimensional locally Euclidean space endowed with some $\Ck$-structure and $\Uset\subset\Mman$ is an open subset homeomorphic with $\bR$, then there exists a $\Ck$-compatible chart $\Umap\colon\Uset\to\bR$ being a homeomorphism onto, see Corollary~\ref{cor:adding_one_chart}.
In Section~\ref{sect:Ck_ctruct_on_L} we classify $\Ck$-structures on $\DLine$, and, in particular, prove Theorem~\ref{th:Ck-struct_on_L_simple}, see Theorem~\ref{th:Ck-struct_on_L_detailed}.
Finally in Section~\ref{th:proof:th:ck_structs_examples} we prove Theorem~\ref{th:ck_structs_examples}.

\section{Preliminaries}\label{sect:preliminaries}
Throughout the paper the arrows $\monoArrow$ and $\epiArrow$ mean \term{monomorphism} and \term{epimorphism} respectively.
We will also regard the cyclic group $\bZ_{2}$ of order $2$ as the multiplicative group $\{\pm1\}$.

\subsection{Semidirect products}\label{sect:semidirect_products}
Let $\Hgrp$ be a group with a unit $e$ and $\Agrp$ be some subgroup of $\Hgrp$.
A homomorphism $s\colon\Hgrp\to\Agrp$ is called \term{retraction} if $s(a)=a$ for all $a\in\Agrp$.

Suppose $s\colon\Hgrp\to\Agrp$ is a homomorphism-retraction, and let $\Kgrp$ be its kernel.
Then $\Agrp$ acts on $\Kgrp$ by conjugations, and this allows to define the following group $\Kgrp\rtimes_{s}\Agrp$ called the \term{semidirect product} of $\Kgrp$ and $\Agrp$ corresponding to the retraction $s$.
Namely, $\Kgrp\rtimes_{s}\Agrp$ is a cartesian product of \term{sets} $\Kgrp\times\Agrp$ with the following multiplication:
\[
    (\kel,\ael)(\lel,\bel) := (\kel \ael\lel\ael^{-1},\ael\bel),  \qquad (\kel,\ael),(\lel,\bel) \in\Kgrp\times\Agrp.
\]
Moreover, the map $\gamma\colon\Kgrp\rtimes_{s}\Agrp \to \Hgrp$, $\gamma(\kel,\ael)=\kel\ael$, is an isomorphism of groups making commutative the following diagram:
\[
  \xymatrix@C=5em@R=1.2em{
    \Kgrp \ar@{^(->}[r]^-{\kel\,\mapsto\,(\kel, e)} \ar@{=}[d] &
    \Kgrp\rtimes_{s}\Agrp   \ar@{->>}[r]^-{(\kel,\ael)\,\mapsto\,\ael} \ar[d]_-{\gamma}^-{\cong} &
    \Agrp \ar@{=}[d]
    \\
    \Kgrp  \ar@{^(->}[r]     &
    \Hgrp  \ar@{->>}[r]^-{s} &
    \Agrp
  }
\]
We will exploit this construction several times.

\subsection{Wreath product $\Dgrp\wr\bZ_2$}\label{sect:wreath_prod}
Let $\Dgrp$ be a group and
\[
    \phi\colon\Dgrp\times\Dgrp\to\Dgrp\times\Dgrp,
    \qquad
    \phi(a,b)=(b,a),
\]
be the automorphism of exchanging coordinates.
It has the order $2$, and therefore defines a monomorphism $\bar{\phi}\colon\bZ_{2}\to \Aut(\Dgrp\times\Dgrp)$, given by $\bar{\phi}(1)=\id_{\Dgrp\times\Dgrp}$ and $\bar{\phi}(-1)=\phi$.

The corresponding semidirect product $(\Dgrp\times\Dgrp)\rtimes\bZ_2$ is called the \term{wreath product} of $\Dgrp$ and $\bZ_{2}$ and denoted by $\Dgrp\wr\bZ_2$, see e.g.~\cite{Meldrum:1995}.
Thus, $\Dgrp\wr\bZ_{2}$ is the direct product of \term{sets} $\Dgrp\times\Dgrp\times\bZ_{2}$ with the following binary operation:
\[
    (a,b,\delta) (c,d,\eps) =
    \begin{cases}
        (ac,bd, \delta\eps), & \delta = +1, \\
        (ad,bc, \delta\eps), & \delta = -1.
    \end{cases}
\]

\subsection{Double cosets $\gCHD$}\label{sect:double_cosets}
We recall here the definition of double cosets, see e.g.~\cite[Section~1.7]{Hall:Gr:1959}.
Let $\Hgrp$ be a group and $\Cgrp$, $\Dgrp$ be two subgroups.
Then we have a natural \term{left} action of the product $\Cgrp\times\Dgrp$ on $\Hgrp$:
\[
    \Cgrp\times\Dgrp\times\Hgrp\to\Hgrp,
    \qquad
    (\cel,\del)\cdot \hel = \cel\hel\del^{-1}.
\]
The corresponding orbit of $\hel\in\Hgrp$ is called the \term{$(\Cgrp,\Dgrp)$-double coset} of $\hel$ and is denoted by $\Cgrp\hel\Dgrp$.
If $\Cgrp=\Dgrp$, then $\Dgrp\hel\Dgrp$ is also called the \term{$\Dgrp$-double coset} of $\hel$.

Notice that
\[
    \Cgrp\hel\Dgrp
        = \{ \cel\hel\del^{-1}\mid \cel\in\Cgrp, \del\in\Dgrp\}
        = \{ \cel\hel\del\mid \cel\in\Cgrp, \del\in\Dgrp\}
        = \mathop{\cup}\limits_{\cel\in\Cgrp} \cel\hel\Dgrp
        = \mathop{\cup}\limits_{\del\in\Dgrp} \Cgrp\hel\del.
\]
In particular, it is the union of left cosets of $\Dgrp$ corresponding to elements of the right coset $\Cgrp\hel$, as well as the union of right cosets of $\Cgrp$ corresponding to elements of the left coset $\hel\Dgrp$.
The partition of $\Hgrp$ into $(\Cgrp,\Dgrp)$-double cosets is denoted by
\[ \gCHD. \]
Partitions into double cosets are not so ``regular'' as partitions into left or right cosets, and they are often hard to compute, see e.g.~\cite{Seitz:DblCoset:1998, HamiltonWiltonZalesskii:JLMS:2013, Hamilton:ProcAMS:2016, Kuan:AHP:2019}, and references therein.
For instance, the double coset classes may have distinct sizes.
The reason is that the above action of $\Cgrp\times\Dgrp$ on $\Hgrp$ is not always free.
\begin{subexample}\rm
Let $\bD_3 = \langle r,s \mid r^3=s^2=(rs)^2=1\rangle$ be the dihedral group consisting of six elements $e,r,r^2,s,sr,sr^2$.
Consider two subgroups $A = \{ e, s \}$ and $B = \{ e, s r \}$ of $\bD_3$.
One easily checks that
\begin{gather*}
    \dbl{A}{\bD_3}{B} = \bigl\{ \{e,s,r,sr\}, \ \{ r^2, sr^2\} \bigr\}, \\
    \dbl{A}{\bD_3}{A} = \dbli{A}{\bD_3}  = \bigl\{ \{e,s\}, \{ r, r^2, sr, sr^2\} \bigr\}.
\end{gather*}
\end{subexample}

The following lemma is trivial.
We will refer to it several times, and therefore it will be convenient to put it as a separate statement:
\begin{sublemma}\label{lm:DHD_normal}
  If $\Cgrp=\Dgrp$ and this group is normal in $\Hgrp$, then $\gCHD$ coincides with the partition $\Hgrp/\Dgrp$ into \textup{(}no matter left or right\textup{)} cosets of $\Dgrp$ in $\Hgrp$.
\qed
\end{sublemma}

Let us mention one more description of $\gCHD$ which will be useful in this paper.
\begin{sublemma}\label{lm:CHD_partition}
Let $\gel,\hel\in\Hgrp$.
Then the following conditions are equivalent:
\begin{enumerate}[itemsep=1ex]
\item\label{enum:lm:CHD_partition:CgD=ChD} $\gel\in\Cgrp\hel\Dgrp$, which is the same as $\Cgrp\gel\Dgrp = \Cgrp\hel\Dgrp$;
\item\label{enum:lm:CHD_partition:C_gDh_non_empty} $\Cgrp\cap \gel\Dgrp\hel^{-1}  \ne \varnothing$.
\end{enumerate}
\end{sublemma}
\begin{proof}
\ref{enum:lm:CHD_partition:CgD=ChD}$\Rightarrow$\ref{enum:lm:CHD_partition:C_gDh_non_empty}
Suppose $\gel\in\Cgrp\hel\Dgrp$, so $\gel = \cel\hel\del^{-1}$ for some $\cel\in\Cgrp$ and $\del\in\Dgrp$.
Then $\cel = \gel\del\hel^{-1} \in \Cgrp\cap\gel\Dgrp\hel^{-1}  \ne \varnothing$.

\ref{enum:lm:CHD_partition:C_gDh_non_empty}$\Rightarrow$\ref{enum:lm:CHD_partition:CgD=ChD}
Conversely, if $\Cgrp\cap\gel\Dgrp\hel^{-1}$ contains some element $\cel$, i.e.\ $\cel = \gel\del\hel^{-1}$ for some $\del\in\Dgrp$, then $\gel = \cel\hel\del^{-1} \in \Cgrp\hel\Dgrp$.
\end{proof}

\subsection{Double cosets $\dbli{\Dgrp}{\Hgrp}$}\label{sect:pm_double_cosets}
Suppose $\Cgrp=\Dgrp$.
Then for each $\hel\in\Hgrp$, we have that $(\Dgrp\hel\Dgrp)^{-1} = \Dgrp\hel^{-1}\Dgrp$ is the $\Dgrp$-double coset of $\hel^{-1}$, see e.g.\ the discussion in~\cite[\S3]{Frame:BAMS:1941}.
The following set
\[
    \Dgrp\hel^{\pm1}\Dgrp := \Dgrp\hel\Dgrp \cup \Dgrp\hel^{-1}\Dgrp = \{ \cel\hel\del, \cel\hel^{-1}\del \mid \cel,\del\in\Dgrp\}
\]
will be called \term{$(\Dgrp,\pm)$-double coset of $\hel$}.
Thus, if $\Dgrp\hel\Dgrp$ mutually contains some $\gdif\in\Hgrp$ and its inverse, then $\Dgrp\hel^{\pm1}\Dgrp=\Dgrp\hel\Dgrp=\Dgrp\hel^{-1}\Dgrp$.
Otherwise, $\Dgrp\hel^{\pm1}\Dgrp = \Dgrp\hel\Dgrp \sqcup \Dgrp\hel^{-1}\Dgrp$.

It follows that $(\Dgrp,\pm)$-double cosets yield a partition on $\Hgrp$ which will be denoted by
\[
    \dbli{\Dgrp}{\Hgrp}.
\]

One can give another description of $\dbli{\Dgrp}{\Hgrp}$.
Note that we have a natural \term{left} action
\[
    (\Dgrp\wr\bZ_{2}) \times\Hgrp \to \Hgrp,
    \qquad
    (a,b,\delta)\cdot\hel := a \hel^{\delta} b^{-1}.
\]
Indeed, if $a,b,c,d\in\Dgrp$ and $\hel\in\Hgrp$, then
\begin{align*}
(a,b,1) \bigl( (c,d,1)\cdot \hel \bigr)
    &= a(c \hel d^{-1})b^{-1}
     = (ac,bd, 1)\cdot\hel
     = \bigl( (a,b,1) (c,d,1) \bigr) \cdot \hel,
\\
(a,b,-1) \bigl( (c,d,1)\cdot \hel \bigr)
    &= a(c \hel d^{-1})^{-1} b^{-1}
     = (ad, bc, -1) \cdot\hel
     = \bigl( (a,b,-1) (c,d,1) \bigr)\cdot \hel,
\\
(a,b,1) \bigl( (c,d,-1)\cdot \hel \bigr)
    &= ac \hel^{-1} d^{-1} b^{-1}
     = (ac, bd, -1) \cdot\hel
     = \bigl( (a,b,1) (c,d,-1) \bigr)\cdot \hel,
\\
(a,b,-1) \bigl( (c,d,-1)\cdot \hel \bigr)
    &= a(c \hel^{-1} d^{-1})^{-1} b^{-1}
     = (ad, bc, 1)\cdot \hel
     = \bigl( (a,b,-1) (c,d,-1) \bigr)\cdot \hel.
\end{align*}
It is easy to see that $\dbli{\Dgrp}{\Hgrp}$ coincides with the partition into orbits of that action of $\Dgrp\wr\bZ_2$.

\subsection{Branch points}
Let $\Mman$ be a topological space with topology $\topX$, and $\px\in\Mman$ be a point.
Denote by $\hcl{\px}:=\mathop{\cap}\limits_{\px\in\Wman\in\topX} \Cl{\Wman}$ the intersection of closures of all open neighborhoods of $\px$, we will call it the \term{Hausdorff closure of $\px$}.
Evidently, $\px\in\hcl{\px}$.
Also note that $\hcl{\px}$ consists of $\py\in\Mman$ such that $\px$ and $\py$ have no disjoint neighborhoods.

Say that $\px$ is a \term{branch} point of $\Mman$.
It is well known and is easy to see that $\Mman$ is Hausdorff iff $\{\px\}=\hcl{\px}$ for all $\px\in\Mman$, i.e.\ $\Mman$ has no branch points.
There are not standard terms for such points.
For instance in~\cite{OConnel:TA:2024} they are called \term{Hausdorff-violating}.

The following simple lemma is left for the reader.
\begin{sublemma}\label{lm:branch_pt_invariant}
Let $\dif\colon\Mman\to\Mman$ be a homeomorphism of $\Mman$.
Then $\dif(\hcl{\px})=\hcl{\dif(\px)}$ for all $\px\in\Mman$.
In particular, the set of branch points of $\Mman$ is invariant under homeomorphisms.
\end{sublemma}

\section{Differentiable structures on locally euclidean spaces}\label{sect:Ck_structs_on_loc_eu_spaces}
This section briefly recalls definition of differentiable structures on locally Euclidean spaces.
We refer the reader to the book by J.~Lee~\cite[Chapters~1,2]{Lee:SmoothManif:2013}, and also to the paper of F.~Takens~\cite{Takens:BSMF:1979} for an equivalent approach.

Usually a \term{manifold} is a Hausdorff second countable locally Euclidean topological space.
The first two conditions are necessary and sufficient for metrizability of manifolds and for the  possibility to embed them into some Euclidean spaces.
However, for defining differentiable structures only the property of being locally Euclidean is enough.
Therefore we will assume here only that property.
For simplicity we also do not consider manifolds with boundary.

For each $m\geq0$ the real $m$-dimensional linear space $\bR^{m}$ with the usual topology is called \term{Euclidean}.
Note that $\bR^{0}=\{0\}$ is a one-point space being a $0$-dimensional linear space.
Also recall that a homeomorphism $\dif\colon\Uset\to\Vset$ between open subset of $\bR^{m}$ is called a \term{$\Ck$-diffeomorphism} for some $\kk\in\bN\cup\{\infty\}$, if both $\dif$ and $\dif^{-1}$ are $\Ck$-maps.

\subsection{Locally Euclidean spaces}
Let $\Mman$ be a topological space.
A \term{chart} on $\Mman$ is an open embedding (i.e.\ a homeomorphism onto open subset) $\Umap\colon\Uset\to\bR^{m}$ of some open subset $\Uset\subset\Mman$ for some $m\geq0$.
In this case we will say that this chart is \term{$m$-dimensional} and $\Uset$ is its \term{domain}.
It will often be convenient to regard such a chart as a pair $\Uchr=(\Uset,\Umap)$.

Notice that then for every open subset $\Uset' \subset \Uset$ the restriction $\restr{\Umap}{\Uset'}\colon\Uset'\to\bR^{m}$ is also an $m$-dimensional chart on $\Mman$.

Any collection $\UAtlas=\{(\Useti,\Umapi)\}_{\ui\in\UInd}$ of charts on $\Mman$ will be called a \term{partial atlas}.
If their domains cover $\Mman$, i.e.\ $\Mman=\mathop{\cup}\limits_{\ui\in\UInd}\Useti$, then $\UAtlas$ is called an \term{atlas}.
Evidently, the union of any collection of (partial) atlases is a (partial) atlas.
It will also be convenient to regard each chart $\Uchr$ as a partial atlas $\{\Uchr\}$ consisting of that chart only.
This will never lead to confusion.

Say that $\Mman$ is \term{locally Euclidean} if either of the following equivalent conditions holds:
\begin{enumerate}[label={\rm(LE\arabic*)}, itemsep=1ex]
\item\label{enum:manif:rr}
every point $\px\in\Mman$ has an open neighborhood homeomorphic to some open subset of some Euclidean space (whose dimension may depend on $\px$);
\item\label{enum:manif:cover}
$\Mman$ admits an atlas.
\end{enumerate}

\subsection{Differentiable structures}
Let $\Uchr=(\Uset,\Umap)$ and $\Vchr=(\Vset,\Vmap)$ be two \term{overlapping} charts on $\Mman$ with values in $\bR^{m}$, i.e.\ $\Uset\cap\Vset \ne \varnothing$.
Then the following well-defined homeomorphism between open subsets of $\bR^{m}$
\begin{equation}\label{equ:trans_map}
    \Vmap\circ\Umap^{-1}\colon\Umap(\Uset\cap\Vset)\to\Vmap(\Uset\cap\Vset)
\end{equation}
is called the \term{transition map} from $\Uchr$ to $\Vchr$.
Note that its inverse, $\Umap\circ\Vmap^{-1}\colon\Vmap(\Uset\cap\Vset)\to\Umap(\Uset\cap\Vset)$, is then the transition map from $\Vchr$ to $\Uchr$.

In what follows fix once and for all some $\kk\in\bN\cup\{\infty\}$.
Then the charts $\Uchr$ and $\Vchr$ are \term{$\Ck$-compatible} if either $\Uset\cap\Vset=\varnothing$ or the transition map $\Vmap\circ\Umap^{-1}$ is a $\Ck$-diffeomorphism.

A partial atlas $\UAtlas=\{(\Uset_i,\Umap_i)\}_{i\in\UInd}$ on $\Mman$ is \term{$\Ck$} if all its charts are pair-wise $\Ck$-compatible.

Further, two partial $\Ck$-atlases $\UAtlas$ and $\VAtlas$ are \term{$\Ck$-compatible} if their union $\UAtlas \cup \VAtlas$ is still a partial $\Ck$-atlas, i.e.\ the charts from $\UAtlas$ and $\VAtlas$ are pair-wise $\Ck$-compatible as well.

Notice that the relation \term{``to be $\Ck$-compatible''} on the set of all \term{partial atlases} on $\Mman$ is not an equivalence: it is reflexive and symmetric, but not transitive.
However, it is easy to show that it is transitive on the set of all \term{atlases}:
\begin{sublemma}
The relation \term{``to be $\Ck$-compatible''} on the set of all atlases on $\Mman$ is an equivalence relation.
In particular, the union of any collection of pair-wise $\Ck$-compatible $\Ck$-atlases is still a $\Ck$-atlas.
\end{sublemma}
The corresponding equivalence classes of pair-wise $\Ck$-compatible atlases on $\Mman$ are called \term{$\Ck$-structures}.
It will be convenient to say that an atlas $\UAtlas$ on $\Mman$ \term{represents} a $\Ck$-structure $\Ustruct$ if $\UAtlas\in\Ustruct$.
Also, a partial atlas $\UAtlas$ on $\Mman$, e.g.\ a chart, is \term{$\Ck$-compatible} with $\Ustruct$ if it is $\Ck$-compatible with some atlas from $\Ustruct$.

Note also that taking the union of all atlases from the same $\Ck$-structure $\Ustruct$, we will get a \term{maximal atlas} of $\Ustruct$, in the sense that it contains any other $\Ck$-compatible atlas.
Hence, each \term{$\Ck$-structure is represented by a unique maximal atlas}.
However, such a description is not very useful from the ``practical'' point of view, since the maximal $\Ck$-atlases are too large to be ``\term{observable}''.
In many cases it is more convenient to work with atlases having as small number of charts as possible.
See also Lemma~\ref{lm:char_smooth_structures} below for another characterization of differentiable structures.

Moreover, the number of \term{$\Ck$-structures} on $\Mman$ is still too large, and in order to get some reasonable classification one should consider those structures up to a homeomorphism being $\Ck$-diffeomorphism between them.

\subsection{Orientations}
The transition map~\eqref{equ:trans_map} \term{preserves orientation} if the determinant of the Jacobi matrix of $\Vmap\circ\Umap^{-1}$ is positive at each point $\px\in\Umap(\Uset\cap\Vset)$.
A $\Ck$-atlas $\UAtlas=\{(\Uset_i,\Umap_i)\}_{i\in\UInd}$ on $\Mman$ is called \term{orientable} if all its transition maps preserve orientation.
A $\Ck$-structure on $\Mman$ is \term{orientable}, if it contains an orientable atlas.

\subsection{Differentiable maps between locally Euclidean spaces}
Till the end of this section we assume that $\Mman$ and $\Nman$ are locally Euclidean spaces endowed with some $\Ck$\nobreakdash-structures $\Ustruct$ and $\Vstruct$ respectively.

Let $\dif\colon\Mman\to\Nman$ be a continuous map and $\Uchr=(\Uset,\Umap)$ and $\Vchr=(\Vset,\Vmap)$ be charts in $\Mman$ and $\Nman$ respectively such that $\dif(\Uman)\subset\Vman$.
Then the induced map between open subsets of Euclidean spaces:
\begin{equation}\label{equ:local_presentation}
    \Vmap\circ\dif\circ\Umap^{-1} \colon \Umap(\Uset) \to \Vmap(\Vset)
\end{equation}
is called a \term{\textup{(}local\textup{)} presentation} of $\dif$ in the charts $\Uchr$ and $\Vchr$.

Say that $\dif$ is \term{$\Ck$ \textup(with respect to $\Ustruct$ and $\Vstruct$\textup)}, if for each $\px\in\Mman$ there exist local charts $\Uchr=(\Uset,\Umap)$ on $\Mman$ and $\Vchr=(\Vset,\Vmap)$ on $\Nman$ such that
\begin{itemize}[itemsep=1ex]
\item $\px\in\Uset$ and $\Uchr$ is $\Ck$-compatible with $\Ustruct$;
\item $\dif(\px)\in\Vset$ and $\Vchr$ is $\Ck$-compatible with $\Vstruct$;
\item $\dif(\Uset) \subset \Vset$ and the presentation $\Vmap\circ\dif\circ\Umap^{-1} \colon \Umap(\Uset) \to \Vmap(\Vset)$ of $\dif$ in $\Uchr$ and $\Vchr$ is $\Ck$ at the point $\Umap(\px)$.
\end{itemize}
In this case we will also say use the notation $\dif\colon(\Mman,\Ustruct)\to(\Nman,\Vstruct)$ to indicate that $\dif$ is a $\Ck$-map with respect to $\Ustruct$ and $\Vstruct$, and denote by $\CkABMN$ the set of all such maps.

A homeomorphism $\dif\colon\Mman\to\Nman$ is a \term{$\Ck$-diffeomorphism (with respect to $\Ustruct$ and $\Vstruct$)} or \term{between $(\Mman,\Ustruct)$ and $(\Nman,\Vstruct)$}, if $\dif\in\CrAB{\kk}{\Ustruct}{\Vstruct}{\Mman}{\Nman}$ and $\dif^{-1}\in\CrAB{\kk}{\Vstruct}{\Ustruct}{\Nman}{\Mman}$.

The set of all $\Ck$-diffeomorphisms $\dif\colon(\Mman,\Ustruct)\to(\Nman,\Vstruct)$ will be denoted by $\DiffkABMN$.
In particular, if $\Mman=\Nman$, and $\Ustruct=\Vstruct$, then $\DrAB{\kk}{\Ustruct}{\Ustruct}{\Mman}{\Mman}$ will also be abbreviated  to $\DiffkAM$ or even to $\Diff^{\kk}(\Mman)$ if the corresponding $\Ck$-structure is assumed from the context.

The following easy lemma gives another characterization of $\Ck$-structures.

\begin{sublemma}\label{lm:char_smooth_structures}
Two $\Ck$-structures $\Ustruct$ and $\Vstruct$ on $\Mman$ coincide iff $\id_{\Mman} \in\DrAB{\kk}{\Ustruct}{\Vstruct}{\Mman}{\Mman}$, i.e.\ the identity map of $\Mman$ is a $\Ck$-diffeomorphism with respect to them.
\end{sublemma}

\begin{subexample}\rm
Let $\Uset\subset\bR^{m}$ be an open subset.
Then $\Uset$ admits an atlas $\CanonAtlas{\Uset} := \{(\Uset, \id_{\Uset})\}$ consisting of a unique chart being just the inclusion map $\id_{\Uset}\colon\Uset\subset\bR^{m}$.
This atlas is $\Ck$ for all $k\in\bN\cup\{\infty\}$.
The corresponding $\Ck$-structure on $\Uman$ will be called \term{canonical} and denoted by $\CanonStr{\Uset}$.

Also, if $\Vset\subset\bR^{n}$ is another open subset, then a map $\dif\colon\Uset\to\Vset$ belongs to
$\CrAB{\kk}{\CanonStr{\Uset}}{\CanonStr{\Vset}}{\Uset}{\Vset}$ iff $\dif$ is $\Ck$ as a map between open subsets of Euclidean spaces in the usual sense.
\end{subexample}

\begin{subexample}\rm
Let $\UAtlas = \{ (\Useti,\Umapi) \}_{\ui\in\UInd}$ be a $\Ck$-atlas on $\Mman$ representing a $\Ck$-structure $\Ustruct$, and $\Xman\subset\Mman$ be an open subset.
Then the following statements hold.
\begin{enumerate}[itemsep=1ex]
\item $\Xman$ is locally Euclidean as well;
\item $\restr{\UAtlas}{\Xman} := \{(\Xman\cap\Useti, \restr{\Umapi}{\Xman\cap\Useti}) \mid \Xman\cap\Useti \ne \varnothing, \ui\in\UInd\}$ is a $\Ck$-atlas on $\Xman$;
\item
Denote by $\restr{\Ustruct}{\Xman}$ the $\Ck$-structure on $\Xman$ represented by $\restr{\UAtlas}{\Xman}$.
Then the inclusion map $i\colon\Xman\subset\Mman$ is $\Ck$ with respect to $\restr{\Ustruct}{\Xman}$ and $\Ustruct$.
\end{enumerate}

Moreover, let $\Vchr=(\Vset,\Vmap)$ be any $m$-dimensional chart $\Ck$-compatible with $\Ustruct$.
Thus $\Vmap(\Vset)$ is an open subset of $\bR^{m}$, and so it has a canonical $\Ck$-structure $\CanonStr{\Vmap(\Vset)}$ given by the atlas $\CanonAtlas{\Vmap(\Vset)}$.
Then the corresponding homeomorphism $\Vmap\colon\Vset\to\Vmap(\Vset)$ is a $\Ck$-difeomorphism with respect to the $\Ck$-structures $\restr{\Ustruct}{\Vman}$ and $\CanonStr{\Vmap(\Vset)}$.
\end{subexample}

\begin{sublemma}\label{lm:act_homeo_on_atlases}
Let $\dif\colon\Mman\to\Nman$ be a homeomorphism, $0\leq \kk \leq \infty$, and $\Uchr=(\Uset,\Umap)$ and $\Vchr=(\Vset,\Vmap)$ be charts on $\Mman$.
Then the following statements hold.
\begin{enumerate}[label={\rm(\arabic*)}, leftmargin=*, itemsep=1ex]
\item
$\dif^{*}(\Uchr):=(\dif(\Uman),\Umap\circ\dif^{-1})$ is a chart on $\Mman$ with the same image $\Umap(\Uset)$, and the presentation of $\dif$ in the charts $\Uchr$ and $\dif^{*}(\Uchr)$ is the identity map of $\Umap(\Uset)$:
\[
    (\Umap\circ\dif^{-1}) \circ \dif \circ \Umap^{-1} = \id_{\Umap(\Uset)};
\]

\item
The transition maps from $\Uchr$ to $\Vchr$ and from $\dif^{*}(\Uchr)$ to $\dif^{*}(\Vchr)$ coincide, i.e.
\[
    (\Vmap\circ\dif^{-1}) \circ (\Umap\circ\dif^{-1})^{-1} = \Vmap\circ\Umap^{-1};
\]

\item
Let $\UAtlas=\{(\Useti,\Umapi)\}_{\ui\in\UInd}$ be a $\Ck$-atlas on $\Mman$ representing a $\Ck$-structure $\Ustruct$.
Then
\[
    \dif^{*}\UAtlas=\{ (\dif(\Useti),\Umapi\circ\dif^{-1})\}_{\ui\in\UInd}
\]
is an atlas on $\Nman$ with the same transition maps between the charts with the corresponding indices.
In particular, the following statements hold.
\begin{enumerate}[label={\rm(\alph*)}, ref={\rm(3\alph*)}, itemsep=1ex]
\item $\dif^{*}\UAtlas$ is of class $\Ck$ as well, and we will denote its $\Ck$-structure by $\dif^{*}\Ustruct$;
\item\label{enum:HM_acts_on_smooth_struct}
$\UAtlas$ is a maximal $\Ck$-atlas iff so is $\dif^{*}\UAtlas$;
\item
If $\kk\geq1$, then $\dif$ is a \term{$\Ck$-diffeomorphism} with respect to $\Ustruct$ and $\dif^{*}\Ustruct$, i.e.\ $\dif\in\DrAB{\kk}{\Ustruct}{\dif^{*}\Ustruct}{\Mman}{\Mman}$.
\end{enumerate}
\end{enumerate}
\end{sublemma}
Thus, due to Lemma~\ref{lm:act_homeo_on_atlases}\ref{enum:HM_acts_on_smooth_struct}, for each $\kk\in\bN\cup\{\infty\}$ the group $\Homeo(\Mman)$ of homeomorphisms of $\Mman$ naturally acts on the set of $\Ck$-structures, and these homeomorphisms are diffeomorphisms between the corresponding $\Ck$-structures.
Therefore the \term{right} formulation of the problem of classification of $\Ck$-structures on $\Mman$ reduces to the following one:
\begin{subproblem}
Describe the orbits of the action of the group $\Homeo(\Mman)$ on the set of $\Ck$-structures on $\Mman$.
\end{subproblem}

In particular, the statement that \term{each manifold $\Mman$ of dimension $\leq 3$ has a unique \textup(up to a diffeomorphism\textup) $\Ck$-structure} means that the above action is transitive.
In next section we will present an elementary proof of that fact for $\dim\Mman=1$.
Such a statement is usually referred as a classical result, though the authors does not find its proof in the available literature.
In fact, we prove some technical statement which will be useful for studying non-Hausdorff manifolds.

Let us mention few simple statements which will be useful lated in the paper for the study of differentiable structures on non-Hausdorff $1$-manifolds.
\begin{sublemma}\label{lm:atlases_with_one_chart}
Let $\Xman$ be an open subset of $\bR^{n}$ for some $n\geq0$, and $\Ustruct$ and $\Vstruct$ be two $\Ck$-structures on $\Xman$ represented by atlases $\{(\Xman,\Umap)\} \in \Ustruct$ and $\{(\Xman,\Vmap)\} \in \Vstruct$ consisting of single charts and such that $\Umap(\Xman) = \Vmap(\Xman) \subset\bR^{n}$.
Then the homeomorphism
\[
    \Vmap^{-1}\circ\Umap\colon
        \Xman \xrightarrow{\Umap} \Umap(\Xman) = \Vmap(\Xman) \xrightarrow{\Vmap^{-1}}\Xman
\]
is a $\Ck$-diffeomorphism between $(\Xman,\Ustruct)$ and $(\Xman,\Vstruct)$, i.e.\ $\Vmap^{-1}\circ\Umap\in\DrAB{\kk}{\Ustruct}{\Vstruct}{\Xman}{\Xman}$.
\end{sublemma}
\begin{proof}
Consider the following commutative diagram:
\[
    \xymatrix{
        \Xman \ar[r]^-{\Vmap^{-1}\circ\Umap}  \ar[d]_-{\Umap} & \Xman \ar[d]^-{\Vmap} \\
        \Umap(\Xman) \ar@{=}[r]^-{\id} & \Vmap(\Xman)
    }
\]
It means that the local presentation of $\Vmap^{-1}\circ\Umap$ in the charts $\Uchr$ and $\Vchr$ is the identity map $\id\colon\Umap(\Xman) = \Vmap(\Xman)$ which is a diffeomorphism of class $\Ck$.
Hence $\Vmap^{-1}\circ\Umap$ is a $\Ck$\nobreakdash-diffeomorphism.
\end{proof}

\begin{sublemma}\label{lm:DiffM_DiffN}
Let $\Mman$ and $\Nman$ be two locally Euclidean spaces endowed with $\Ck$-structres $\Ustruct$ and $\Vstruct$ respectively.
Suppose there exists a $\Ck$-diffeomorphism $\dif\in\DiffkABMN$.
Then we have the following commutative diagram of bijections:
\[
\xymatrix{
    & \DiffkABMN \ar[rd]^-{\ \kdif\,\mapsto\,\kdif\circ\dif^{-1}}
    \\
    \DrA{\kk}{\Ustruct}{\Mman}
    \ar[ur]^-{\gdif\,\mapsto\,\dif\circ\gdif \ }
    \ar[rr]^-{\gdif \, \mapsto \, \dif\circ\gdif\circ\dif^{-1}} &&
    \DrA{\kk}{\Vstruct}{\Nman}
}
\]
in which the lower horizontal arrow is an isomorphism of groups.
In particular, if the groups $\DrA{\kk}{\Ustruct}{\Mman}$ and $\DrA{\kk}{\Vstruct}{\Nman}$ are not isomorphic, then $(\Mman,\Ustruct)$ and $(\Nman,\Vstruct)$ are not $\Ck$-diffeomorphic.
\end{sublemma}

The inverse statement, i.e.\ that every isomorphism $\alpha\colon\DrA{\kk}{\Ustruct}{\Mman} \to \DrA{\kk}{\Vstruct}{\Nman}$ is induced by some diffeomorphism between $\Mman$ and $\Nman$, is called a \term{reconstrucion theorem}.
It holds for Hausdorff second countable manifolds $\Mman$ and $\Nman$, see~\cite{Whittaker:AnnMath:1963, Rubin:TrAMS:1989, KowalikMichalikRybicki:TrMath:2008, RubinRybicki:TA:2012, ChenMann:DMJ:2023} for precise statements and references therein.

\section{Uniqueness of $\Ck$-structures on $\bR$ up to a diffeomorphism}\label{sect:diff_struct_on_bR}
In this section $\Mman$ is a one-dimensional locally Euclidean space and $k\in\bN\cup\{\infty\}$.
We will describe a way of joining charts on $\Mman$.
This technique is mentioned in~\cite{Sawin:MO}.

\begin{definition}
Let $\Uchr=(\Uset,\Umap)$ and $\Vchr=(\Vset,\Vmap)$ be two charts on $\Mman$.
Say that $\Uchr$ and $\Vchr$ are \term{$\Ck$-joinable} if they are $\Ck$-compatible and
\begin{align*}
&\Umap(\Uset)=(a;c), &
&\Umap(\Uset\cap\Vset)=\Vmap(\Uset\cap\Vset)=(b;c), &
&\Vmap(\Vset)=(b;d)
\end{align*}
for some $a<b<c<d\in\bR$.

In this case every open embedding $\Wmap\colon\Uset\cup\Vset\to\bR$, regarded as a chart $\Wchr=(\Uset\cup\Vset,\Wmap)$ on $\Mman$, will be called a \term{$\Ck$-join} of $\Uchr$ and $\Vchr$ if, see Figure~\ref{fig:joinable_charts},
\begin{enumerate}
\item $\Wmap(\Uset) = \Umap(\Uset) = (a;c)$ and $\Wmap(\Vset) = \Vmap(\Vset) = (b;d)$;
\item $\Wchr$ is $\Ck$-compatible with $\Uchr$ and $\Vchr$;
\item $\restr{\Wmap}{\Uset\setminus\Vset} = \restr{\Umap}{\Uset\setminus\Vset} \colon \Uset\setminus\Vset\to(a;c]$ and $\restr{\Wmap}{\Vset\setminus\Uset} = \restr{\Vmap}{\Vset\setminus\Uset}\colon \Vset\setminus\Uset\to[b;d)$.
\end{enumerate}
\end{definition}

\begin{figure}[!htbp]
\includegraphics[height=5cm]{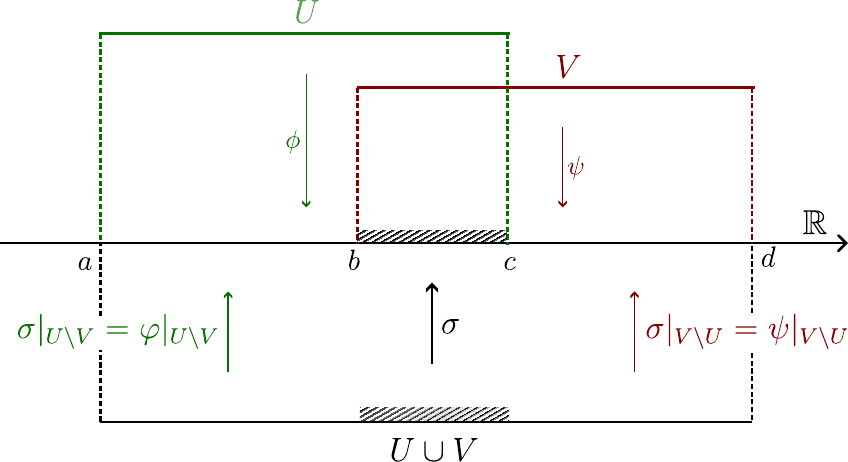}
\caption{$\Ck$-joinable charts}\label{fig:joinable_charts}
\end{figure}

The first condition implies that $\Wmap(\Uset\cap\Vset) = (c;b)$, while the latter condition can be formulated so that \term{$\Wmap$ differs from $\Umap$ and $\Vmap$ only on the intersection $\Uset\cap\Vset$}, or equivalently, that the transition maps $\Wmap\circ\Umap^{-1}$ and $\Wmap\circ\Vmap^{-1}$ are identity on $\Uset\setminus\Vset$ and $\Vset\setminus\Uset$ respectively.

Our aim is to establish the following theorem
\begin{theorem}\label{th:joinable_charts}
\begin{enumerate}[leftmargin=*, itemsep=1ex]
\item\label{th:joinable_charts:join_exist}
Every pair of $\Ck$-joinable charts $\Uchr$ and $\Vchr$ has a $\Ck$-join.

\item\label{th:joinable_charts:can_replace}
Let $\UAtlas$ be a $\Ck$-atlas on $\Mman$ having a pair of $\Ck$-joinable charts $\Uchr$ and $\Vchr$ such that for every other chart $\Ychr=(\Yset,\Ymap)$ from $\UAtlas$ we have that
\[
    \Yset\cap\Uset\cap\Vset = \varnothing.
\]
Then for any $\Ck$-join $\Wchr$ of $\Uchr$ and $\Vchr$ the following collection of charts
\[ \VAtlas = \bigl(  \UAtlas \setminus \{ \Uchr, \Vchr \} \bigr) \cup \{ \Wchr\}, \]
obtained by replacing $\Uchr$ and $\Vchr$ with $\Wchr$, is still a $\Ck$-atlas on $\Mman$ being also $\Ck$-compatible with $\UAtlas$.
\end{enumerate}
\end{theorem}
\begin{proof}
First we establish the following simple lemma.
\newcommand\va{b}
\newcommand\vb{c}
\begin{sublemma}\label{lm:glue_id_and_diff}
Let
$\va<\vb\in\bR$ and
$\gdif\colon(\va;\vb)\to(\va;\vb)$ be a $\Ck$-diffeomorphism preserving orientation.
Then there exist $\eps>0$ and a diffeomorphism $p\colon(\va;\vb)\to(\va;\vb)$ such that
\begin{align*}
    p(x)&=x        \ \text{for} \ x\in(\va;\va+\eps], &
    p(x)&=\gdif(x) \ \text{for} \ x\in[\vb-\eps;\vb).
\end{align*}
\end{sublemma}
\begin{proof}
It suffices to find $\eps>0$ and a strictly positive function $\gamma\colon[\va;\vb)\to(0;+\infty)$ such that
\begin{align*}
&\gamma(x)=1         \ \text{for} \ x\in[\va;\va+\eps], &
&\gamma(x)=\gdif'(x) \ \text{for} \ x\in[\vb-\eps;\vb), &
&\smallint_{\va}^{\vb}\gamma(s)ds = \vb-\va,
\end{align*}
and then put $p(x) = \va + \int_{\va}^{\px}\gamma(s)ds$.

As $\gdif$ is a preserving orientation homeomorphism of $(\va;\vb)$, we have that
\[ \va = \lim\limits_{\px\to \va+0}\gdif(\px) <  \lim\limits_{\px\to \vb-0}\gdif(\px) = \vb. \]
Then one can easily find $\eps>0$ and two $\Cinfty$ functions $\alpha,\beta\colon[\va;\vb]\to[0;+\infty)$ such that
\begin{gather*}
\alpha=1 \ \text{on} \ [\va;\va+\eps], \qquad
\alpha=0 \ \text{on} \ [\va+2\eps;\vb-2\eps],   \qquad
\alpha=\gfunc' \ \text{on} \ [\vb-\eps;\vb], \\
A:=\vb-\va - \smallint_{\va}^{\vb}\alpha(s)ds > 0, \qquad
\beta^{-1}(0)=[\va;\va+\eps]\cup[\vb-\eps;\vb].
\end{gather*}
Denote $B = \smallint_{\va}^{\vb}\beta(s)ds$.
Then the function $\gamma = \alpha + (A/B) \beta$ has the above properties.
\end{proof}

Now we can prove Theorem~\ref{th:joinable_charts}.

\ref{th:joinable_charts:join_exist}
Let $\Uchr=(\Uset,\Umap)$ and $\Vchr=(\Vset,\Vmap)$ be two $\Ck$-joinable charts on $\Mman$.
Not loosing generality, assume that
\begin{align*}
&\Umap(\Uset)=(\ua;\uc), &
&\Umap(\Uset\cap\Vset)=\Vmap(\Uset\cap\Vset)=(\ub;\uc), &
&\Vmap(\Vset)=(\ub;\ud).
\end{align*}
Denote $\Wset=\Uset\cup\Vset$.
We need to construct a $\Ck$-join $\Wchr=(\Wset, \Wmap)$ for these charts.

Let $\UVTrMap=\Vmap\circ\Umap^{-1}\colon(\ub;\uc)\to(\ub;\uc)$ be the transition map from $\Uchr$ to $\Vchr$.
Then by Lemma~\ref{lm:glue_id_and_diff} one can find $\eps>0$ and a $C^k$-diffeomorphism $p\colon(\ua;\uc)\to(\ua;\uc)$ such that
\begin{align}
\label{equ:join:p_id_near_a}
p(\px) &= \px                \ \text{for} \ \px\in(\ua;\ub+\varepsilon), \\
\label{equ:join:p_is_g_near_c}
p(\px) &= \UVTrMap(\px)      \ \text{for} \ \px\in(\uc-\varepsilon;\uc).
\end{align}

Consider the homeomorphism $\tUmap := p \circ \Umap\colon\Uman\to(\ua;\uc)$ regarded as a chart $\tUchr=(\Uset,\tUmap)$.
Then by~\eqref{equ:join:p_id_near_a}, $\tUmap=\Umap$ on $(\ua;\ub+\eps)$, whence
\[
    \tUmap(\Uset\cap\Vset)=\Umap(\Uset\cap\Vset)=\Vmap(\Uset\cap\Vset) = (\ub;\uc).
\]
Moreover, by~\eqref{equ:join:p_is_g_near_c}, the transition map from $\Vchr$ to $\tUchr$, being a $\Ck$-diffeomorphism
\[
    q := \tUmap \circ \Vmap^{-1} = p\circ \Umap \circ \Vmap^{-1}  = p\circ\UVTrMap^{-1} \colon (\ub;\uc)\to(\ub;\uc),
\]
is fixed on $(\uc-\eps';\uc)$ for a sufficiently small $\eps'>0$.
Hence $q$ extends by the identity on $(\uc;\ud)$ to a self-diffeomorphism of $(\ub;\ud)$, which we will denote by the same letter $q$.

Since $\restr{\tUmap}{\Uset\cap\Vset} = q\circ\restr{\Vmap}{\Uset\cap\Vset}$, we get a wel-defined homeomorphism
\[
    \Wmap:\Wset=\Uset\cup\Vset\to(\ua;\ud),
    \qquad
    \Wmap(\px) =
    \begin{cases}
        \tUmap(\px),      & \px\in\Uset, \\
        q\circ\Vmap(\px), & \px\in\Vset,
    \end{cases}
\]
i.e.\ the chart $\Wchr=(\Wset, \Wmap)$.

Note that $p$ and $q$ are the transition maps from $\Uchr$ and $\Vchr$ to $\Wchr$ respectively:
\begin{align*}
\Wmap\circ\Umap^{-1} &= \tUmap \circ \Umap^{-1} = p\circ \Umap\circ\Umap^{-1} = p\colon(\ua;\uc)\to(\ua;\uc), \\
\Wmap\circ\Vmap^{-1} &= q\circ\Vmap \circ \Vmap^{-1} = q\colon(\ub;\ud)\to(\ub;\ud).
\end{align*}
Since they are $\Ck$-diffeomorphisms, $p$ is fixed on $\Uset\setminus\Vset$, and $q$ is fixed on $\Vset\setminus\Uset$, we see that $\Wchr$ is a $\Ck$-join of $\Uchr$ and $\Vchr$.

\ref{th:joinable_charts:can_replace}
Let $\Uchr=(\Uset,\Umap)$ and $\Vchr=(\Vset,\Vmap)$ be two $\Ck$-joinable charts in the atlas $\UAtlas$ such that for every other chart $\Ychr=(\Yset,\Ymap)$ from $\UAtlas$ we have that $\Yset\cap\Uset\cap\Vset = \varnothing$.
Let also $\Wchr$ be any $\Ck$-join of $\Uchr$ and $\Vchr$.
We need to prove that
\[ \VAtlas = \bigl(  \UAtlas \setminus \{ \Uchr, \Vchr \} \bigr) \cup \{ \Wchr\}, \]
is a $\Ck$-atlas on $\Mman$ being also $\Ck$-compatible with $\UAtlas$.

Since $\Wset=\Uset\cup\Vset$, the domains of charts of $\VAtlas$ cover $\Mman$.

Let us show that \term{$\Wchr$ is $\Ck$-compatible with all other charts $\Ychr=(\Yset,\Ymap) \in \VAtlas$ distinct from $\Wchr$}, i.e.\ the transition map $\Wmap\circ\Ymap^{-1}\colon\Ymap(\Yset\cap\Wset)\to\Wmap(\Yset\cap\Wset)$ is a $\Ck$-diffeomorphism.

Indeed, since $\Yset\cap\Uset\cap\Vset = \varnothing$, we have that $\Yset\cap\Wset$ is a union of disjoint open subsets $\Yset\cap\Uman$ and $\Yset\cap\Vman$, i.e.\ $\Yset\cap\Wset = (\Yset\cap\Uset) \sqcup (\Yset\cap\Vset)$.
Since $\Wmap$ coincides with $\Umap$ on $\Uset\setminus\Vset \supset \Yset\cap\Uset$ and with
$\Vmap$ on $\Vset\setminus\Uset \supset \Yset\cap\Vset$, we see that
\begin{align}\label{equ:trans_maps:w_y}
    &
    \restr{\Wmap\circ\Ymap^{-1}}{\Ymap(\Yset \cap \Uset)} =
    \restr{\Umap\circ\Ymap^{-1}}{\Ymap(\Yset \cap \Uset)},
    &
    &
    \restr{\Wmap\circ\Ymap^{-1}}{\Ymap(\Yset \cap \Vset)} =
    \restr{\Vmap\circ\Ymap^{-1}}{\Ymap(\Yset \cap \Vset)}.
\end{align}
But $\UAtlas$ is a $\Ck$-atlas, i.e.\ $\Ychr$ is $\Ck$-compatible with $\Uchr$ and $\Vchr$, whence both restrictions~\eqref{equ:trans_maps:w_y} are $\Ck$, and therefore $\Wmap\circ\Ymap^{-1}$ is $\Ck$ as well.
\end{proof}

\subsection{Chain-like atlases}\label{sect:chain_like_atlases}
Let $\VAtlas = \{ \Vchri=(\Vseti,\Vmapi) \}_{\vi\in\bZ}$ be a countable $\Ck$-atlas on the real line $\bR$.
Say that $\VAtlas$ is \term{chain-like} if for each $\vi\in\bZ$ the following conditions hold:
\begin{enumerate}[label={\rm(G\arabic*)}, itemsep=1ex]
\item\label{enum:def:chain_like_atlas:joinable}
the charts $\Vchrii{\vi}$ and $\Vchrii{\vi+1}$ are $\Ck$-joinable;

\item\label{enum:def:chain_like_atlas:triple_inters}
$\Vsetii{\ui} \cap \Vsetii{\vi} \cap  \Vsetii{\vi+1} = \varnothing$ for all $\ui \ne \vi,\vi+1$.
\end{enumerate}

Equivalently, suppose that $\Vmapi(\Vseti)=(\aii{\vi};\bii{\vi})$ for some $\aii{\vi}<\bii{\vi}\in\bR\cup\{\pm\infty\}$, $\vi\in\bZ$.
Then $\VAtlas$ is \term{chain-like} if for every $\vi\in\bZ$ we have that
\begin{itemize}[itemsep=1ex]
\item all $\aii{\vi},\bii{\vi}$ are finite,
\item $\aii{\vi} < \bii{\vi-1} < \aii{\vi+1} < \bii{\vi}$,
\item $\Vmapii{\vi}(\Vsetii{\vi}\cap\Vsetii{\vi+1})=\Vmapii{\vi+1}(\Vsetii{\vi}\cap\Vsetii{\vi+1})=(\aii{\vi+1};\bii{\vi})$.
\end{itemize}

Notice that for any pairs of strictly increasing two-sided sequences $\{ x_i\}_{i\in\bZ}$ and $\{y_i\}_{i\in\bZ}$ in $\bR$ there exists a $\Cinfty$-diffeomorphism $q\colon\bR\to\bR$ such that $q(x_i)=y_i$ for $i\in\bZ$.
Hence one can always assume that for such a chain-like atlas we have that
\begin{align}\label{equ:lim_ai_bi}
    &\lim_{\vi\to-\infty} \aii{\vi} = \lim_{\vi\to-\infty} \bii{\vi} = -\infty,
    &
    &\lim_{\vi\to+\infty} \aii{\vi} = \lim_{\vi\to+\infty} \bii{\vi} = +\infty,
\end{align}
which is equivalent to the assumption that
\begin{enumerate}[label={\rm(G\arabic*)}, resume, itemsep=1ex]
\item\label{enum:def:chain_like_atlas:infty}
$\mathop{\cup}\limits_{\vi\in\bZ} \Vmapii{\vi}(\Vsetii{\vi}) = \bR$.
\end{enumerate}

\begin{subtheorem}\label{th:atlases_on_R}
\begin{enumerate}[label={\rm(\arabic*)}, leftmargin=*, itemsep=1ex]
\item\label{enum:th:atlases_on_R:chain_like_exist}
For every $\Ck$-atlas $\UAtlas=\{(\Useti, \Umapi)\}_{\ui\in\UInd}$ on $\bR$ there exists a chain-like $\Ck$-atlas $\VAtlas$ being $\Ck$-compatible with $\VAtlas$.

\item\label{enum:th:atlases_on_R:one_chart}
Let $\VAtlas$ be a chain-like $\Ck$-atlas on $\bR$.
Then there exists a homeomorphism $\Wmap\colon\bR\to\bR$ such that the chart $(\bR,\Wmap)$ is $\Ck$-compatible with $\VAtlas$.
\end{enumerate}
\end{subtheorem}
\begin{proof}
\ref{enum:th:atlases_on_R:chain_like_exist}
Since $\bR$ is one-dimensional and paracompact, and the domains of $\UAtlas$ constitute an open cover of $\bR$, one can find a countable $\{\Vseti\}_{\vi\in\bZ}$ cover of $\bR$ such that for every $\vi\in\bZ$
\begin{enumerate}[label={(\alph*)}, leftmargin=*, itemsep=1ex]
\item\label{enum:lm:unique_Ck_struct:R:intersect_disjoint}
$\Vseti = (\aii{\vi};\bii{\vi})$ for some $\aii{\vi} < \bii{\vi}\in\bR$ such that $\aii{\vi} < \bii{\vi-1} < \aii{\vi+1} < \bii{\vi}$, i.e.\ $\Vseti$ intersects only $\Vsetii{\vi-1}$ and $\Vsetii{\vi+1}$ and those intersections are proper subsets of $\Vseti$, so in particular,
\begin{equation}\label{equ:triple_intersections}
    \Vsetii{\ui} \cap \Vsetii{\vi}\cap\Vsetii{\vi+1} = \varnothing, \qquad \ui \ne \vi,\vi+1;
\end{equation}
\item\label{enum:lm:unique_Ck_struct:R:finer_cover}
there exists a chart $\Uchrii{\vi}=(\Usetii{\vi},\Umapii{\vi}) \in \UAtlas$ with $\Vsetii{\vi} \subset \Usetii{\vi}$;
\end{enumerate}

Note that for any $\Ck$-embdding $p_{\vi}\colon\Umapii{\vi}(\Vsetii{\vi})\to\bR$, the chart $\Vchrii{\vi}:=(\Vsetii{\vi},\Vmapii{\vi}:=p_{\vi}\circ\restr{\Umapii{\vi}}{\Vsetii{\vi}})$ is still $\Ck$-compatible with $\UAtlas$.
In particular, one can choose $p_{\vi}$ so that
\begin{align}\label{equ:fixing_ends_of_Vi}
    &\Vmapii{\vi}(\Vsetii{\vi})                      = (\aii{\vi}   ;  \bii{\vi}), &
    &\Vmapii{\vi}(\Vsetii{\vi-1} \cap\Vsetii{\vi})   = (\aii{\vi}   ;  \bii{\vi-1}) &
    &\Vmapii{\vi}(\Vsetii{\vi}   \cap\Vsetii{\vi+1}) = (\aii{\vi+1} ;  \bii{\vi}),
\end{align}
which will imply that $\Vchrii{\vi}$ and $\Vchrii{\vi+1}$ are $\Ck$-compatible for all $\vi\in\bZ$.

Indeed, note $\Umapii{\vi}(\Vsetii{\vi}) = (\cii{\vi}; \dii{\vi})$ for some $\cii{\vi} < \dii{\vi} \in \bR$.
Take $p_{\vi}\colon(\cii{\vi}; \dii{\vi}) \to (\aii{\vi}; \bii{\vi})$ to be any $\Ck$-diffeomorphism such that $p_{\vi}\circ\Umapii{\vi}(\bii{\vi-1}) = \bii{\vi-1}$ and $p_{\vi}\circ\Umapii{\vi}(\aii{\vi+1}) = \aii{\vi+1}$.
Then the composite map
\[
    \Vmapii{\vi} := p_{\vi}\circ\restr{\Umapii{\vi}}{\Vsetii{\vi}}\colon
        \Vsetii{\vi}=(\aii{\vi}; \bii{\vi})
        \xrightarrow{~\Umapii{\vi}~}
        (\cii{\vi}; \dii{\vi})
        \xrightarrow{~p_{\vi}~}
        (\aii{\vi}; \bii{\vi})
        \subset
        \bR
\]
has the desired property~\eqref{equ:fixing_ends_of_Vi}.

Then~\eqref{equ:fixing_ends_of_Vi} and~\eqref{equ:triple_intersections} mean respectively conditions~\ref{enum:def:chain_like_atlas:joinable} and~\ref{enum:def:chain_like_atlas:triple_inters}, so
$\VAtlas:=\{ (\Vsetii{\vi},\Vmapii{\vi}) \}_{\vi\in\bZ}$ is a chain-like $\Ck$-atlas on $\bR$.

\ref{enum:th:atlases_on_R:one_chart}
Suppose now that $\bR$ has a chain-like atlas $\VAtlas = \{ \Vchri=(\Vseti,\Vmapi) \}_{\vi\in\bZ}$.
As noted above, one can assume that condition~\eqref{equ:lim_ai_bi} is also satisfied.

Then by Theorem~\ref{th:joinable_charts} there exists a $\Ck$-join $\Wchr_{0,1}=(\Vsetii{0} \cup \Vsetii{1}, \Wmapii{0,1})$ of  $\Vchrii{0}$ and $\Vchrii{1}$, such that the collection of charts
\[
    \VAtlas_{0,1} := \bigl( \VAtlas  \setminus\{ \Vchrii{0}, \Vchrii{1} \} \bigr) \cup \{\Wchrii{0,1}\}
\]
is a $\Ck$-atlas being $\Ck$-compatible with $\VAtlas$.
By definition, $\Wchrii{0,1}$ differs from $\Vmapii{0}$ and $\Vmapii{1}$ only on the intersection $\Vsetii{0} \cap \Vsetii{1}$ which is disjoint from other chars from $\VAtlas$.
This easily implies that $\VAtlas_{0,1}$ is still chain-like.

In particular, the charts $\Vchrii{-1}$ and $\Wchr_{0,1}$ are $\Ck$-joinable, and again by Theorem~\ref{th:joinable_charts} for any of their $\Ck$-join
\[
    \Wchr_{-1,1} =
    \bigl(
        \Vsetii{-1} \cup \Wsetii{0,1} = \Vsetii{-1} \cup \Vsetii{0} \cup \Vsetii{1}, \
        \Wmapii{-1,1}
    \bigr)
\]
the collection of charts
\[
    \VAtlas_{-1,1} := \bigl( \VAtlas  \setminus\{ \Vchrii{-1}, \Wchrii{0,1} \} \bigr) \cup \{\Wchrii{-1,1}\}
    =
    \bigl( \VAtlas  \setminus\{ \Vchrii{-1}, \Vchrii{0}, \Vchrii{1} \} \bigr) \cup \{\Wchrii{-1,1}\}
\]
is a chain-like $\Ck$-atlas being $\Ck$-compatible with $\VAtlas_{0,1}$ and therefore with $\VAtlas$.
Again notice that $\Wmapii{-1,1}$ differs from with $\Vmapii{-1}$ and $\Wmapii{0,1}$ only on the intersection
\[ \Vsetii{-1} \cap(\Vsetii{0} \cap \Vsetii{1})=\Vsetii{-1} \cap\Vsetii{0}.\]
In particular, $\Wmapii{-1,1}=\Vmapii{0}$ on the compact segment
\[
    \Aman_1:=\Vsetii{0}\setminus(\Vsetii{-1}\cup\Vsetii{1}) = \Vmapii{0}^{-1}([\bii{-1}, \aii{1}]).
\]

Thus, by induction for each $n\geq1$ we can construct a sequence of chain-like $\Ck$-atlases
\[
    \VAtlas_{-n,n} := \bigl( \VAtlas  \setminus\{ \Vchrii{\vi} \}_{\vi=-n}^{n} \bigr) \cup \{\Wchrii{-n,n}\}
\]
where $\Wchrii{-n,n}=\bigl(\mathop{\cup}\limits_{\vi=-n}^{n}\Vsetii{\vi}, \Wmapii{-n,n}\bigr)$ is a chart such that $\Wmapii{-n,n}\bigl(\mathop{\cup}\limits_{\vi=-n}^{n}\Vsetii{\vi} \bigr) = (\aii{-n};\bii{n})$ and $\Wmapii{-n,n}=\Wmapii{-n+1,n-1}$ on the compact segment
\[
    \Aman_{n}:=
    \bigl(\mathop{\cup}\limits_{\vi=-n}^{n}\Vsetii{\vi}\bigr)
    \setminus
    \bigl(
        \Vsetii{-n} \cup \Vsetii{n}
    \bigr).
\]

Since $\Aman_{n} \subset \Int{\Aman_{n+1}}$ and $\bR = \mathop{\cup}\limits_{n=0}^{+\infty}\Aman_{n}$, for each $\px\in\bR$ there exists $n>0$ such that
\[
    \Wmapii{-i,i}(\px)=\Wmapii{-j,j}(\px) \ \text{for all} \ i,j\geq n.
\]
Then due to~\eqref{equ:lim_ai_bi}, we get a well-defined homeomorphism $\Wmap\colon\bR\to\bR$, $\Wmap(\px) = \lim\limits_{n\to\infty}\Wmapii{-n,n}(\px)$, which is $\Ck$-compatible with all the above atlases and, in particular, with $\VAtlas$.
\end{proof}

\begin{corollary}\label{cor:Ck_struct_on_R}
Every $\Ck$-structure $\Vstruct$ on $\bR$ is $\Ck$-diffeomorphic with the canonical one given by the atlas $\CanonAtlas{\bR}=\{(\bR,\id_{\bR})\}$ with a single chart.
\end{corollary}
\begin{proof}
By Theorem~\ref{th:atlases_on_R}, $\Vstruct$ is represented by an atlas $\VAtlas=\{(\bR,\Vmap)\}$ consisting of a single chart given by some homeomorphism $\Vmap\colon\bR\to\bR$.
On the other hand, the canonical $\Ck$-structure $\CanonStr{\bR}$ is also given by the atlas with a single chart whose coordinate homeomorphism $\id\colon\bR\to\bR$ has the same image $\bR$ as $\Vmap$.
Then by Lemma~\ref{lm:atlases_with_one_chart}, the map $\Vmap = \Vmap\circ\id_{\bR}^{-1}\colon\bR\to\bR$ is a $\Ck$-diffeomorphism between $(\bR,\Vstruct)$ and $(\bR, \CanonStr{\bR})$.
\end{proof}

\begin{subcorollary}\label{cor:adding_one_chart}
Let $\Mman$ be a one-dimensional locally Euclidean topological space, $\UAtlas$ a $\Ck$-atlas on $\Mman$, and $\Uman\subset\Mman$ be an open subset homeomorphic with $\bR$.
Then there exists a homeomorphism $\Umap\colon\Uman\to\bR$ being $\Ck$-compatible with $\UAtlas$.
In other words, $\UAtlas \cup \{ (\Uman, \Uman)\}$ is a $\Ck$-atlas on $\Mman$ being $\Ck$-compatible with $\UAtlas$.
\end{subcorollary}

We finish this section with the following well-known statement which could be proved in using the approach of similar to Theorem~\ref{th:atlases_on_R}.
\begin{theorem}\label{th:uniqueness_of_ck-struct}
Let $\Mman$ be a connected Hausdorff one-dimensional locally Euclidean topological space.
Then for any two $\Ck$-structures $\Ustruct$ and $\Vstruct$ on $\Mman$ there exists a $\Ck$-diffeomor\-phism $\dif\in\DrAB{\kk}{\Ustruct}{\Vstruct}{\Mman}{\Mman}$.
In other words, the action of the group of homeomorphisms $\Homeo(\Mman)$ on the set of all $\Ck$-structures on $\Mman$ is transitive.
\end{theorem}

\section{Differentiable structures on the line with double origin}\label{sect:Ck_ctruct_on_L}
It this section we will give a classification of $\Ck$-structures on the line with two origins $\DLine = \bR\sqcup\{\ptb\}$ for all $\kk=1,2,\ldots,\infty$.
It will be convenient to describe at first the structure of the groups $\HomeoL$ of all homeomorphisms of $\DLine$, and then compare it with the subgroups of diffeomorphisms of distinct differentiable structures.
Denote $\Na := \bR = \DLine\setminus\ptb$ and $\Nb := \DLine\setminus\pta$.

\subsection{Homeomorphisms of $\bR$ and $\DLine$}
\newcommand\calG{\mathcal{G}}
Consider the diffeomorphism $\qdif\colon\bR\to\bR$, $\qdif(\px)=-\px$, reversing orientation and having order $2$.
Evidently, $\qdif\in\JDiffRz[\infty,-]$.

Let $\calG$ be one of the groups $\HomeoRzOrPres$, $\HRzOrPres$, $\DiffRzOrPres$ or $\JDiffRzOrPres$ defined in Section~\ref{sect:notation}, and $\calG^{+} = \calG\cap\HomeoRzOrPres$.

Then it is easy to see that $\qdif\,\calG^{+}\qdif^{-1}=\calG^{+}$, so the conjugation by $\qdif$ is an automorphism of $\calG^{+}$ of order $2$.
Hence one can define the semidirect product $\calG^{+}\rtimes\bZ_{2}$ corresponding to that automorphism.
Thus, by definition, $\calG^{+}\rtimes\bZ_{2}$ is a cartesian product $\calG^{+}\times\bZ_{2}$ with the following multiplication $(\dif,\delta)(\gdif,\eps) = (\dif\circ\qdif^{\delta}\circ\gdif\circ\qdif^{-\delta}, \delta\eps)$.
In particular, the inclusion $\calG^{+}\subset \HomeoRzOrPres$ allows to consider $\calG^{+}\rtimes\bZ_{2}$ as a subgroup of $\HomeoRzOrPres\rtimes\bZ_{2}$.

The following two lemmas are straightforward, see Section~\ref{sect:semidirect_products}, and we leave them for the reader.
They describe certain algebraic properties of the groups of homeomorphisms of $\bR$ and $\DLine$ and also establish relationships between them.
\begin{sublemma}
The map $\eta\colon\HomeoRzOrPres\rtimes\bZ_{2} \to \HomeoRzOrPres$, $\eta(\dif,\delta) = \dif\circ\qdif^{\delta}$, is an isomorphism making commutative the following diagram:
\[
  \xymatrix@C=5em{
    \HomeoRzOrPres \ar@{^(->}[r]^-{\gdif\,\mapsto\,(\gdif,0)} \ar@{=}[d] &
    \HomeoRzOrPres \rtimes_{\qdif} \bZ_2   \ar@{->>}[r]^-{(\gdif,a)\,\mapsto\,a}
        \ar[d]_-{\eta}^-{\cong} &
    \bZ_2 \ar@{=}[d]
    \\
    \HomeoRzOrPres  \ar@{^(->}[r]  &
    \HomeoRz \ar@{->>}[r]^{s} &
    \langle\qdif\rangle\cong\HomeoRz/\HomeoRzOrPres
  }
\]
and isomorphically mapping $\calG^{+}\times\bZ_{2}$ onto $\calG$.
\end{sublemma}

\begin{sublemma}\label{lm:HomeoL}
\begin{enumerate}[leftmargin=*, label={\rm(\arabic*)}, itemsep=2ex]
\item\label{enum:lm:HomeoL:xi2_id}
Let $\invol\colon\DLine\to\DLine$ be the (unique) map exchanging $\pta$ and $\ptb$ and fixed on the complement $\DLine\setminus\BranchPtDLine$.
Then
\begin{enumerate}[label={\rm(\alph*)}, ref={\rm(2\alph*)}, itemsep=1ex]
\item $\invol$ is a homeomorphism of $\DLine$ and $\invol^2=\id_{\DLine}$, so it generates the subgroup $\grpInvol\cong\bZ_2$;

\item $\invol$ commutes with each $\dif\in\HomeoL$ and $\grpInvol$ coincides with the center of $\HomeoL$.
\end{enumerate}

\item
$\HomeoLxOFix$ is the kernel of the ``restriction to $\BranchPtDLine$'' homomorphism
\[
    \perm\colon\HomeoL\to\mathrm{Perm}(\BranchPtDLine)=\bZ_2,
    \qquad
    \perm(\dif)=\restr{\dif}{\BranchPtDLine},
\]
to the group of permutations of $\BranchPtDLine$.
Moreover, the ``restriction to $\Na$'' map
\[
    \rmap\colon\HomeoLxOFix \to \HomeoRz,
    \qquad
    \rmap(\dif)=\restr{\dif}{\Na},
\]
is an isomorphism of groups.
The inverse map $\rmap^{-1}\colon\HomeoRz\to\HomeoLxOFix$ associates to each $\dif\in\HomeoRz$ its unique extension to a homeomorphism of all $\DLine$ fixing $\ptb$.

\item\label{enum:lm:HomeoL:H_xi}
The normal subgroups $\HomeoLxOFix$ and $\grpInvol$ generate $\HomeoL$, and $\HomeoLxOFix \cap \grpInvol = \{\id_{\DLine}\}$, whence $\HomeoL$ splits into the direct product $\HomeoLxOFix \times \grpInvol$.
More explicitly, the map
\[
    \alpha\colon\HomeoLxOFix \times \grpInvol \equiv \HomeoRz \times \bZ_2 \to \HomeoL,
    \qquad
    \alpha(\gdif, \eps) = \gdif\circ\xi^{\eps},
\]
is an isomorphism of groups making commutative the following diagram:
\[
  \xymatrix@C=5em@R=2em{
    \HomeoLxOFix \ar@{^(->}[r] \ar[d]_-{\rmap\colon\dif\,\mapsto\,\restr{\dif}{\Na}}^-{\cong} &
    \HomeoL     \ar@{->>}[r]^-{\perm\colon\dif\,\mapsto\,\restr{\dif}{\BranchPtDLine}}  &
    \mathrm{Perm}(\BranchPtDLine) \ar@{=}[d]
    \\
    \HomeoRz \ar@{^(->}[r]^-{\gdif\,\mapsto\,(\gdif,0)}  &
    \HomeoRz \times \bZ_2 \ar@{->>}[r]^-{(\gdif,\eps)\,\mapsto\,\eps }  \ar[u]^-{\alpha}_-{\cong} &
    \bZ_2
  }
\]

\item
Let $\qdif\in\HomeoRz$ be any homeomorphism such that $\qdif^2=\id_{\bR}$, and $\hat{\qdif} = \rmap^{-1}(\qdif)\in\HomeoLxOFix$ be it unique extension to all of $\DLine$ fixing $\ptb$.
Then we have an isomorphism
\[
    \bigl( \HomeoRzOrPres \rtimes_{\qdif} \bZ_2 \bigr) \times \bZ_2  \,\cong\, \HomeoL,
    \qquad
    (\gdif, a, \eps)  \mapsto \gdif\circ \hat{\qdif}^{a} \circ \invol^{\eps}.
\]
\end{enumerate}
\end{sublemma}

\subsection{Group $\DRzp$}
From now on fix some $\kk\in\bN\cup\{\infty\}$, and let $\DRzp$ be the group of \term{monotone} $\Ck$-diffeomorphisms of $\Rzp$ (i.e.\ either increasing or decreasing).

\begin{subexample}\rm
The diffeomorphism $\dif\colon\Rzp\to\Rzp$, $\dif(\px)=1/\px$, is \term{not monotone} and therefore does not belong to $\DRzp$.
\end{subexample}

Notice that for each $\dif\in\HRz$ its restriction $\restr{\dif}{\Rzp}\colon\Rzp\to\Rzp$ is a monotone $\Ck$-diffeomorphism.
\begin{sublemma}
The map $\gamma\colon\HRz\to\DRzp$, $\gamma(\dif) = \restr{\dif}{\Rzp}$, is an isomorphism of groups.
\end{sublemma}
\begin{proof}
Evidently, $\gamma$ is a monomorphism, and we need to show that $\gamma$ is surjective.
Let $\gdif\in\DRzp$.
Since $\gdif$ is monotone, $\lim\limits_{\px\to0}\gdif(\px)=0$, whence $\gdif$ extends to a homeomorphism $\hat{\gdif}\colon\bR\to\bR$, defined by $\hat{\gdif}=\gdif$ on $\Rzp$ and $\hat{\gdif}(0)=0$.
It follows, that $\hat{\gdif}\in\HRz$ and $\gamma(\hat{\gdif})=\gdif$, i.e.\ $\gamma$ is surjective as well.
\end{proof}

For each $\gdif\in\DRzp$ the homeomorphism $\gamma^{-1}(\gdif)$ will be called the \term{$0$-extension} on $\gdif$.
Note also that we have the following inclusion
\[
    \DiffRz \ \subset  \ \aHRzp \ \stackrel{\gamma^{-1}}{\equiv} \ \DRzp.
\]
Thus, one can regard $\DiffRz$ as a subgroup of $\DRzp$.
It is easy to check that $\DiffRz$ is not normal in $\DRzp$.

\begin{sublemma}\label{lm:Dnb_is_normal}
The group $\aDRnbp$ is normal in $\aHRzp$.
\end{sublemma}
\begin{proof}
Let $\dif\in\aDRnbp$ and $\gdif\in\aHRzp$.
Then $\gdif\circ\dif\circ\gdif^{-1}$ is fixed near $0\in\bR$ and its restriction to $\Rzp$ is a $\Ck$-diffeomorphism.
Hence $\gdif\circ\dif\circ\gdif^{-1}$ is a $\Ck$-diffeomorphism of $\bR$ fixed near $0$, i.e.\ it belongs to $\aDRnbp$.
\end{proof}

\subsection{Minimal atlases on $\DLine$}
A $\Ck$-atlas $\UAtlas$ on $\DLine$ will be called \term{minimal} if it consisting of only two charts $\Umap\colon\Na\to\bR$ and $\Vmap\colon\Nb\to\bR$ being homeomorphisms onto and such that $\Umap(\pta)=\Vmap(\ptb)=0$.

Notice that for every such an atlas, its transition map from $(\Na,\Umap)$ to $(\Nb,\Vmap)$
\[
    \Vmap\circ\Umap^{-1}\colon\bR\setminus0 \to \bR\setminus0
\]
is a \term{monotone} $\Ck$-diffeomorphism of $\Rzp$, i.e.\ it belongs to $\DRzp$.
The corresponding $0$-extension of $\Vmap\circ\Umap^{-1}$ will be denoted by $\Utrmap$, i.e.\
\[
    \Utrmap:=\gamma(\Vmap\circ\Umap^{-1}) \in \aHRzp.
\]
Note also that by definition, $\UAtlas$ is \term{orientable} iff its unique transition map $\Vmap\circ\Umap^{-1}$ preserves orientation.
The latter is, evidently, equivalent to the assumption that $\Utrmap$ preserve orientations.

The following lemma is the reason of our detailed exposition of uniqueness of $\Ck$-structures on $\bR$.
\begin{sublemma}\label{lm:minimal_atlas_exist}
Every $\Ck$-structure $\Ustruct$ on $\DLine$ is represented by some orientable minimal atlas.
\end{sublemma}
\begin{proof}
Since $\Na$ and $\Nb$ are open subsets homeomorphic with $\bR$, it follows from Corollary~\ref{cor:adding_one_chart} that there exist two homeomorphisms onto $\Umap\colon\Na\to\bR$ and $\Vmap\colon\Nb\to\bR$ being $\Ck$-compatible with $\Ustruct$.
Replacing $\Umap$ and $\Vmap$ with $\Umap - \Umap(\pta)$ and $\Vmap - \Vmap(\ptb)$ respectively, one can achieve that $\Umap(\pta)=\Vmap(\ptb)=0$.
Moreover, replacing (if necessary) $\Vmap$ with $-\Vmap$, we can assume that the transition map $\Vmap^{-1}\circ\Umap$ preserves orientation.
As $\DLine = \Na\cup\Nb$, we have that $\UAtlas=\{ (\Na,\Umap),(\Nb,\Vmap) \}$ is the desired orientable minimal $\Ck$-atlas belonging to $\Ustruct$.
\end{proof}

\begin{subtheorem}\label{th:Dfix_Dex}
Let $\UAtlas=\{(\Na,\Umap),(\Nb,\Vmap)\}$ and $\VAtlas=\{(\Na,\tUmap),(\Nb,\tVmap)\}$ be two minimal $\Ck$-atlases on $\DLine$ corresponding to $\Ck$-structures $\Ustruct$ and $\Vstruct$ respectively, and $\Utrmap$ and $\Vtrmap$ be $0$-extensions of their respective transition maps $\Vmap\circ\Umap^{-1}$ and $\tVmap\circ\tUmap^{-1}$.
\begin{enumerate}[label={\rm\arabic*)}, leftmargin=*, itemsep=1ex]
\item\label{enum:th:Dfix_Dex:bij}
Then the following correspondences $\bijDfix$ and $\bijDex$ are well-defined bijections:
\begin{gather*}
\begin{aligned}
&\bijDfix \colon \DiffLxOFix[\Ustruct,\Vstruct] \to \DiffRz \, \cap \, \Vtrmap\,\DiffRz\,\Utrmap^{-1},
&& \bijDfix(\dif) =  \tVmap\circ\dif\circ\Vmap^{-1}, \\
&\bijDex \colon  \DiffLxOEx[\Ustruct,\Vstruct] \to \DiffRz \, \cap \, \Vtrmap\,\DiffRz\,\Utrmap,
&\qquad\qquad&
\bijDex(\dif) = \tUmap\circ\dif\circ\Vmap^{-1}, \\
\end{aligned} \\
\begin{aligned}
&\xymatrix@C=5em{
    \Nb \ar[r]^-{\dif} \ar[d]_-{\Vmap} & \Nb \ar[d]^-{\tVmap} \\
    \bR \ar[r]^-{\bijDfix(\dif)} & \bR
}
&\qquad\qquad&
\xymatrix@C=5em{
    \Nb \ar[r]^-{\dif} \ar[d]_-{\Vmap} & \Na \ar[d]^-{\tUmap} \\
    \bR \ar[r]^-{\bijDex(\dif)} & \bR
}
\end{aligned}
\end{gather*}
Thus, $\bijDfix(\dif)$ is a local presentation of $\dif\in\DiffLxOFix[\Ustruct,\Vstruct]$ with respect to the charts $(\Nb,\Vmap)$ and $(\Nb,\tVmap)$, while $\bijDex(\dif)$ is a local presentation of $\dif\in\DiffLxOEx[\Ustruct,\Vstruct]$ with respect to the charts $(\Nb,\Vmap)$ and $(\Na,\tUmap)$;

\item\label{enum:th:Dfix_Dex:orientation}
If $\tVmap$ and $\Vmap$ mutually preserve or reverse orientations, then
\begin{align*}
    &\bijDfix\bigl(\DiffLxOrPresxOFix[\Ustruct,\Vstruct]\bigr) \subset \DiffRzOrPres, &
    &\bijDfix\bigl(\DiffLxOrRevxOFix[\Ustruct,\Vstruct]\bigr) \subset \DiffRzOrRev.
\end{align*}
Similarly, if $\tUmap$ and $\Vmap$ mutually preserve or reverse orientations, then
\begin{align*}
    &\bijDex\bigl(\DiffLxOrPresxOEx[\Ustruct,\Vstruct]\bigr) \subset \DiffRzOrPres, &
    &\bijDex\bigl(\DiffLxOrRevxOEx[\Ustruct,\Vstruct]\bigr) \subset \DiffRzOrRev;
\end{align*}

\item\label{enum:th:Dfix_Dex:self_bij}
If $\UAtlas=\VAtlas$, so $\Utrmap=\Vtrmap$ and $\Ustruct=\Vstruct$, then
\[ \bijDfix\colon\DiffLxOFix[\Ustruct] \to \DiffRz \, \cap \, \Utrmap\,\DiffRz\,\Utrmap^{-1} \]
is an isomorphism of groups.
Moreover, the group $\DiffLxOrPresxOFix[\Ustruct]$ contains a subgroup isomorphic with $\aDRnbp$, and therefore it is uncountable;

\item\label{enum:th:Dfix_Dex:DAB_pm_fixex}
If $\DiffLx[\Ustruct,\Vstruct]{*,\bullet}$ is non-empty for some $*\in\{\pm\}$ and $\bullet\in\{\mathrm{fix}, \mathrm{ex} \}$, then it is uncountable.
\end{enumerate}
\end{subtheorem}
\begin{proof}
\ref{enum:th:Dfix_Dex:bij}
First we show that $\bijDfix$ is a well-defined bijection.
Let $\dif\in\DiffLxOFix[\Ustruct,\Vstruct]$.
Then $\dif(\Na)=\Na$ and $\dif(\Nb)=\Nb$, whence we have the following commutative diagrams:
\begin{align*}
    &\xymatrix@R=4ex@C=12ex{
        \Na \ar[r]^-{\dif}  \ar[d]_-{\Umap} & \Na \ar[d]^-{\tUmap} \\
        \bR \ar[r]^-{\adif=\tUmap\circ\dif\circ\Umap^{-1}} & \bR
    }
    &
    \qquad\qquad
    &\xymatrix@R=4ex@C=12ex{
        \Nb \ar[r]^-{\dif}  \ar[d]_-{\Vmap} & \Nb \ar[d]^-{\tVmap} \\
        \bR \ar[r]^-{\bdif=\tVmap\circ\dif\circ\Vmap^{-1}} & \bR
    }
\end{align*}
in which the lower arrows $\adif:=\tUmap\circ\dif\circ\Umap^{-1}$ and $\bdif:=\tVmap\circ\dif\circ\Vmap^{-1}$ are local presentations of $\dif$ in the corresponding charts.
We need to show that $\bijDfix(\dif) := \bdif \in \DiffRz \, \cap \, \Vtrmap\,\DiffRz\,\Utrmap^{-1}$ and that correspondence $\dif\mapsto\bdif$ is a bijection.

The assumption that $\dif$ is a $\Ck$-diffeomorphism of $\DLine$ means that $a,b\in\DiffRz$.
Notice that on $\Rzp$ we have the following identity:
\[
   b
   = \tVmap\circ\dif\circ\Vmap^{-1}
   = (\tVmap\circ\tUmap^{-1}) \circ
     (\tUmap\circ\dif\circ\Umap^{-1}) \circ
     (\Vmap\circ\Umap^{-1})^{-1}.
\]
Taking $0$-extensions of these diffeomorphisms we obtain that
$$b = \Vtrmap\circ a \circ \Utrmap^{-1}\in\Vtrmap\,\DiffRz\,\Utrmap^{-1}.$$
But $\bdif\in\DiffRz$ as well, whence $\bdif\in\DiffRz \, \cap \, \Vtrmap\,\DiffRz\,\Utrmap^{-1}$, and thus the map $\bijDfix$ is well-defined.

Moreover, since $\dif(\pta)=\pta$, $\dif$ is uniquely determined by its restriction to $\DLine\setminus\{\pta\}=\Nb$, and therefore by the composition $\bdif=\tVmap\circ\dif\circ\Vmap^{-1}$.
Hence $\bijDfix$ is injective.

Let us prove that $\bijDfix$ is surjective.
Let $\bdif\in\DiffRz \, \cap \, \Vtrmap\,\DiffRz\,\Utrmap^{-1}$, i.e.\
\[
    b = \Vtrmap\circ a \circ \Utrmap^{-1} = \tVmap\circ\tUmap^{-1} \circ a \circ \Umap\circ\Vmap^{-1}
\]
for some $\adif\in\DiffRz$.
The latter is equivalent to the following identity:
\begin{equation}\label{equ:fix:tuau_tvbv}
    \tUmap^{-1}\circ a \circ \Umap = \tVmap^{-1}\circ b \circ \Vmap.
\end{equation}
Define the map
\begin{equation}\label{equ:dif_constr_from_a_b}
    \dif\colon\DLine\to\DLine,
    \qquad
    \dif(\px) =
    \begin{cases}
        \tUmap^{-1}\circ a \circ \Umap(\px), & \px\in\Na, \\
        \tVmap^{-1}\circ b \circ \Vmap(\px), & \px\in\Nb.
    \end{cases}
\end{equation}
Due to~\eqref{equ:fix:tuau_tvbv}, the formulas for $\dif$ agree on $\Na\cap\Nb$, and therefore $\dif$ is a continuous bijection.
Moreover, we have the following commutative diagrams:
\begin{align*}
    &\xymatrix@R=4ex@C=12ex{
        \Na \ar[r]^-{\restr{\dif}{\Na} = \tUmap^{-1}\circ a \circ \Umap}  \ar[d]_-{\Umap} & \Na \ar[d]^-{\tUmap} \\
        \bR \ar[r]^-{a} & \bR
    }
    &
    &\xymatrix@R=4ex@C=12ex{
        \Nb \ar[r]^-{\restr{\dif}{\Nb} = \tVmap^{-1}\circ b \circ \Vmap}  \ar[d]_-{\Vmap} & \Nb \ar[d]^-{\tVmap} \\
        \bR \ar[r]^-{b} & \bR
    }
\end{align*}
which show that the local presentation of $\dif$ in the charts $\Na$ and $\Nb$ are $\Ck$-diffeomorphisms $a$ and $b$ respectively.
Therefore $\dif$ is a $\Ck$-diffeomorphism between $(\DLine,\Ustruct)$ and $(\DLine,\Vstruct)$.
This proves that $\bijDfix$ is a bijection.

The proof for $\bijDex$ is similar.
Let us just mention that if $\dif\in\DiffLxOEx[\Ustruct,\Vstruct]$, then $\dif(\Na)=\Nb$ and $\dif(\Nb)=\Na$, whence we have the following commutative diagrams:
\begin{align*}
    &\xymatrix@R=4ex@C=12ex{
        \Na \ar[r]^-{\dif}  \ar[d]_-{\Umap} & \Nb \ar[d]^-{\tVmap} \\
        \bR \ar[r]^-{\adif=\tVmap\circ\dif\circ\Umap^{-1}} & \bR
    }
    &
    \qquad\qquad
    &\xymatrix@R=4ex@C=12ex{
        \Nb \ar[r]^-{\dif}  \ar[d]_-{\Vmap} & \Na \ar[d]^-{\tUmap} \\
        \bR \ar[r]^-{\bdif=\tUmap\circ\dif\circ\Vmap^{-1}} & \bR
    }
\end{align*}
in which the lower arrows $\adif:=\tVmap\circ\dif\circ\Umap^{-1}$ and $\bdif:=\tUmap\circ\dif\circ\Vmap^{-1}$ are local presentations of $\dif$ in the corresponding charts.
Then $\bijDex$ is defined by $\bijDex(\dif) = \bdif$.
The verification that $\bijDex$ is a well-defined bijection is similar to the case of $\bijDfix$ and we leave it for the reader.

\ref{enum:th:Dfix_Dex:orientation}
This statement directly follows from the definition of bijections $\bijDfix$ and $\bijDex$.

\ref{enum:th:Dfix_Dex:self_bij}
Suppose $\UAtlas=\VAtlas$, so $\Vmap = \tVmap$ and $\Vtrmap=\Utrmap$.
Then $\Utrmap\,\DiffRz\,\Utrmap^{-1}$ is a conjugated to $\DiffRz$ subgroup of $\aHRzp$, whence $\DiffRz \, \cap \, \Utrmap\,\DiffRz\,\Utrmap^{-1}$ is also a subgroup of $\aHRzp$.
Moreover, $\bijDfix(\dif)=\Vtrmap\circ\dif\circ\Vtrmap^{-1}$ for all $\dif\in\aHRzp$, and therefore $\bijDfix$ is a homomorphism of groups.
Since $\bijDfix$ is also a bijection, it is an isomorphism.

Also, since by Lemma~\ref{lm:Dnb_is_normal}, $\aDRnbp$ is normal in $\aHRzp$, we see that
\[
    \aDRnbp \ =       \  \Utrmap\aDRnbp\Utrmap^{-1}
            \ \subset \  \DiffRz \ \cap \ \Utrmap\,\DiffRz\,\Utrmap^{-1}.
\]
Hence $(\bijDLxOFix[\Ustruct,\Ustruct])^{-1}(\aDRnbp)$ is the desired subgroup of $\DiffLxOrPresxOFix[\Ustruct]$ isomorphic to $\aDRnbp$.

\ref{enum:th:Dfix_Dex:DAB_pm_fixex}
It follows from Lemma~\ref{lm:DiffM_DiffN} that for each $\dif\in\DiffLx[\Ustruct,\Vstruct]{*,\bullet}$ we have three bijections:
\[
\xymatrix{
    & \DiffLx[\Ustruct,\Vstruct]{*,\bullet} \ar[rd]^-{\ \kdif\,\mapsto\,\kdif\circ\dif^{-1}}
    \\
    \DiffLxOrPresxOFix[\Ustruct]
    \ar[ur]^-{\gdif\,\mapsto\,\dif\circ\gdif \ }
    \ar[rr]^-{\gdif \, \mapsto \, \dif\circ\gdif\circ\dif^{-1}} &&
    \DiffLxOrPresxOFix[\Vstruct]
}
\]
Hence, if $\DiffLx[\Ustruct,\Vstruct]{*,\bullet}\ne\varnothing$, then it is uncountable since $\DiffLxOFix[\Ustruct]$ is so.
\end{proof}

\begin{subcorollary}\label{cor:DiffAB_expl}
  Let $\kk\in\bN\cup\{\infty\}$ and $\Ustruct$ and $\Vstruct$ be two $\Ck$-structures on $\DLine$.
  \begin{enumerate}[wide=0pt, label={\rm\arabic*)}, itemsep=1ex]
    \item\label{enum:lm:DiffL_AB_fix:exists}
    Then the following conditions are equivalent:
    \begin{enumerate}[label={\rm(\roman*${}^{\mathrm{fix}}$)}, leftmargin=*, itemsep=1ex]
    \item\label{enum:lm:DiffL_AB_fix:exists:diff}
    $\DiffLxOFix[\Ustruct,\Vstruct]\ne\varnothing$;
    \item\label{enum:lm:DiffL_AB_fix:exists:some_ga_gb_equiv}
    there exist minimal atlases $\UAtlas=\{(\Na,\Umap),(\Nb,\Vmap)\} \in\Ustruct$ and $\VAtlas=\{(\Na,\tUmap),(\Nb,\tVmap)\} \in \Vstruct$ such that $\Vtrmap = \bdif \circ\Utrmap\circ \adif^{-1}$ for some $\adif,\bdif\in\DiffRz$, i.e.\ the $0$-extensions $\Utrmap$ and $\Vtrmap$ of their respective transition maps $\Vmap\circ\Umap^{-1}$ and $\tVmap\circ\tUmap^{-1}$ belong to the same $\DiffRz$-double coset in $\aHRzp$;

    \item\label{enum:lm:DiffL_AB_fix:exists:any_ga_gb_equiv}
    property~\ref{enum:lm:DiffL_AB_fix:exists:some_ga_gb_equiv} holds for any pair of minimal atlases $\UAtlas\in\Ustruct$ and $\VAtlas\in\Vstruct$.
    \end{enumerate}

    \item\label{enum:lm:DiffL_AB_ex:exists}
    Similarly, the following conditions are also equivalent:
    \begin{enumerate}[label={\rm(\roman*${}^{\mathrm{ex}}$)}, leftmargin=*, itemsep=1ex]
    \item\label{enum:lm:DiffL_AB_ex:exists:diff}
    $\DiffLxOEx[\Ustruct,\Vstruct]\ne\varnothing$;
    \item\label{enum:lm:DiffL_AB_ex:exists:some_ga_gb_equiv}
    there exist minimal atlases $\UAtlas\in\Ustruct$ and $\VAtlas \in \Vstruct$ such that $\Vtrmap^{-1} = \bdif \circ\Utrmap\circ \adif^{-1}$ for some $\adif,\bdif\in\DiffRz$;

    \item\label{enum:lm:DiffL_AB_ex:exists:any_ga_gb_equiv}
    property~\ref{enum:lm:DiffL_AB_ex:exists:some_ga_gb_equiv} holds for any pair of minimal atlases $\UAtlas\in\Ustruct$ and $\VAtlas\in\Vstruct$.
    \end{enumerate}

    \item\label{enum:lm:DiffL_AB_all:exists}
    Hence, the following conditions are equivalent as well:
    \begin{enumerate}[label={\rm(\roman*)}, leftmargin=*, itemsep=1ex]
    \item\label{enum:lm:DiffL_AB_all:exists:diff}
    $\DiffLx[\Ustruct,\Vstruct]{} \equiv \DiffLxOFix[\Ustruct,\Vstruct] \sqcup \DiffLxOEx[\Ustruct,\Vstruct]  \ne\varnothing$;
    \item\label{enum:lm:DiffL_AB_all:exists:some_ga_gb_equiv}
    there exist minimal atlases $\UAtlas\in\Ustruct$ and $\VAtlas\in\Vstruct$ such that either $\Vtrmap = \bdif \circ\Utrmap\circ \adif^{-1}$ or $\Vtrmap^{-1} = \bdif \circ\Utrmap\circ \adif^{-1}$ for some $\adif,\bdif\in\DiffRz$, i.e.\ $\Utrmap$ and $\Vtrmap$ belong to the same $(\DiffRz,\pm)$-double coset in $\aHRzp$;

    \item\label{enum:lm:DiffL_AB_all:exists:any_ga_gb_equiv}
    property~\ref{enum:lm:DiffL_AB_all:exists:some_ga_gb_equiv} holds for any pair of minimal atlases $\UAtlas\in\Ustruct$ and $\VAtlas\in\Vstruct$.
    \end{enumerate}
  \end{enumerate}
\end{subcorollary}
\begin{proof}
\ref{enum:lm:DiffL_AB_fix:exists}
The implication~\ref{enum:lm:DiffL_AB_fix:exists:any_ga_gb_equiv}$\Rightarrow$\ref{enum:lm:DiffL_AB_fix:exists:some_ga_gb_equiv} is trivial.

\ref{enum:lm:DiffL_AB_fix:exists:diff}$\Rightarrow$\ref{enum:lm:DiffL_AB_fix:exists:any_ga_gb_equiv}.
Suppose $\DiffLxOFix[\Ustruct,\Vstruct]\ne\varnothing$ and let $\UAtlas\in\Ustruct$ and $\VAtlas\in\Vstruct$ be any minimal atlases.
Then by Theorem~\ref{th:Dfix_Dex}, we have a bijection
$$\bijDfix\colon\DiffLxOFix[\Ustruct,\Vstruct] \equiv \DiffRz \, \cap \, \Vtrmap\,\DiffRz\,\Utrmap^{-1}.$$
Hence, the latter intersection is non-empty, i.e.\ $\bdif=\Vtrmap\circ \adif \circ \Utrmap^{-1}$ for some $\adif,\bdif\in\DiffRz$, which is the same as $\Vtrmap = \bdif \circ\Utrmap\circ \adif^{-1}$.

\ref{enum:lm:DiffL_AB_fix:exists:some_ga_gb_equiv}$\Rightarrow$\ref{enum:lm:DiffL_AB_fix:exists:diff}
Suppose there exist some minimal atlases $\UAtlas\in\Ustruct$ and $\VAtlas\in\Vstruct$ such that $\Vtrmap = \bdif \circ\Utrmap\circ \adif^{-1}$ for some $\adif,\bdif\in\DiffRz$.
Then
$$\bdif=\Vtrmap\circ \adif \circ \Utrmap^{-1} \in \DiffRz \, \cap \, \Vtrmap\,\DiffRz\,\Utrmap^{-1} \ne\varnothing.$$
Hence, by Theorem~\ref{th:Dfix_Dex}, $\DiffLxOFix[\Ustruct,\Vstruct]\ne\varnothing$ as well.

The proof of~\ref{enum:lm:DiffL_AB_ex:exists} is similar, and~\ref{enum:lm:DiffL_AB_all:exists} is just a combination of~\ref{enum:lm:DiffL_AB_fix:exists} and~\ref{enum:lm:DiffL_AB_ex:exists}.
\end{proof}

\subsection{Classification of $\Ck$-structures on $\DLine$}
Let $\CkStructs{\DLine}$ be the set of all $\Ck$-structures on $\DLine$.
Then the group $\HomeoL$ of homeomorphisms of $\DLine$ acts on $\CkStructs{\DLine}$, as described in Lemma~\ref{lm:act_homeo_on_atlases}\ref{enum:HM_acts_on_smooth_struct}.
For every $\Ustruct\in\CkStructs{\DLine}$ let
\begin{itemize}[itemsep=1ex]
\item $[\Ustruct]$ be the orbit of $\Ustruct$ with respect to that action of $\HomeoL$;
\item $[\Ustruct]^{\mathrm{fix}}$ be the orbit of $\Ustruct$ with respect to the induced action of the subgroup $\HomeoLxOFix$ consisting of homeomorphisms fixing $\pta$ and $\ptb$.
\end{itemize}
Thus $[\Ustruct]$ consist of all $\Ck$-structures on $\DLine$ being $\Ck$-diffeomorphic to $\Ustruct$, while $[\Ustruct]^{\mathrm{fix}}$ is a subset of $[\Ustruct]$ consisting of $\Ck$-structures being $\Ck$-diffeomorphic to $\Ustruct$ via a diffeomorphism fixing $\pta$ and $\ptb$.

Denote by $\IsoClassesL$ and $\IsoClassesLFix$ the corresponding orbits spaces.
As a consequence we get the following theorem clarifying the bijections mentioned in Theorem~\ref{th:Ck-struct_on_L_simple}.
\begin{subtheorem}\label{th:Ck-struct_on_L_detailed}
There are well-defined \term{bijections}
\begin{itemize}[itemsep=1ex]
\item
$\sigma\colon
    \IsoClassesL
    \ \longrightarrow \
    \dbli{\DiffRz}{\aHRzp}$,
\[
    \sigma([\Ustruct]) = \DiffRz\,\Utrmap^{\pm1}\,\DiffRz,
    \ \text{where $\UAtlas\in\Ustruct$ is any minimal atlas},
\]
\item
$\sigma^{\mathrm{fix}}\colon
    \IsoClassesLFix
    \ \longrightarrow \
    \dbl{\DiffRz}{\aHRzp}{\DiffRz}$,
\[    \sigma^{\mathrm{fix}}([\Ustruct]^{\mathrm{fix}}) = \DiffRz\,\Utrmap\,\DiffRz,
    \ \text{where $\UAtlas\in\Ustruct$ is any minimal atlas},
\]
\end{itemize}
associating to the isomorphism classes $[\Ustruct]$ and $[\Ustruct]^{\mathrm{fix}}$ of a $\Ck$-structure $\Ustruct$ on $\DLine$ respectively $\DiffRz$- and $(\DiffRz,\pm)$-double cosets in $\aHRzp$ of the $0$-extension $\Utrmap$ of the transition map of \term{any} minimal atlas $\UAtlas\in\Ustruct$.
\end{subtheorem}
\begin{proof}
It is just a rephrasing of Corollary~\ref{cor:DiffAB_expl}.
Namely, the statement that $\sigma$ is a bijection means that two $\Ck$-structures $\Ustruct$ and $\Vstruct$ on $\DLine$ are $\Ck$-diffeomorphic iff for any pair of minimal atlases $\UAtlas\in\Ustruct$ and $\VAtlas\in\Vstruct$ the $0$-extensions $\Utrmap$ and $\Vtrmap$ of their respective transition maps belong to the same $(\DiffRz,\pm)$-double coset in $\aHRzp$.
This is exactly the equivalence of~\ref{enum:lm:DiffL_AB_all:exists:diff} and~\ref{enum:lm:DiffL_AB_all:exists:any_ga_gb_equiv} of Corollary~\ref{cor:DiffAB_expl}.

Similarly, existence of the bijection $\sigma^{\mathrm{fix}}$ is the same as the equivalence~\ref{enum:lm:DiffL_AB_fix:exists:diff}$\Leftrightarrow$\ref{enum:lm:DiffL_AB_fix:exists:any_ga_gb_equiv}.
\end{proof}

\subsection{Special minimal atlases}\label{sect:spec_minimal_atlas}
Our next aim is to provide examples of $\Ck$-structures $\Ustruct$ and $\Vstruct$ on $\DLine$ for which either of sets $\DiffLxOFix[\Ustruct,\Vstruct]$ and $\DiffLxOEx[\Ustruct,\Vstruct]$ can be empty or non-empty.
For that reason it will be convenient to work only with $\Ck$-structures admitting \term{special} minimal atlases defined as follows.

Recall that $\Uset=\bR$ and $\Vset=(\bR\setminus\{\pta\})\cup\{\ptb\}$ by the construction of $\DLine$.
Hence, we have the following homeomorphisms: the identity map $\idU\colon\Uset\to\bR$, and
\[
    \idV\colon\Vman\to\bR,
    \qquad
    \idV(\px)=
    \begin{cases}
    0,             & \px=\ptb, \\
    \px=\idU(\px), & \px\in\Vset\setminus\ptb.
    \end{cases}
\]
In particular, $\idU=\idV$ on $\Uset\setminus\pta=\Vset\setminus\ptb$.

Then for each $\dif\in\aHRzp$ one can define the following minimal atlas
\[
    \MinAtlas{\dif}:=\{ (\Uset,\idU), (\Vset, \dif\circ\idV) \}
\]
whose transition map is
\[
    \restr{\dif\circ\idV\circ\idU^{-1}}{\Rzp} = \restr{\dif}{\Rzp}.
\]
Hence, its $0$-extension is $\dif$, i.e.\ $\trmap{\MinAtlas{\dif}} = \dif$.

The atlas $\MinAtlas{\dif}$ will be called the \term{special minimal atlas of $\dif\in\aHRzp$}.
Denote by $\trStruct{\dif}$ the corresponding $\Ck$-structure on $\DLine$ defined by the atlas $\MinAtlas{\dif}$, so it is determined by $\dif$.
The following lemma collects simple facts about special atlases.
\begin{sublemma}\label{lm:spec_minimal_atlas}
\begin{enumerate}[leftmargin=*, label={\rm(\arabic*)}, itemsep=1ex]
\item\label{enum:spec_minimal_atlas:minimal_umap_id__spec}
Let $\UAtlas=\{ (\Uset,\Umap), (\Vset, \Vmap) \}$ be a minimal atlas on $\DLine$ such that $\Umap=\idU$.
Then $\UAtlas$ is a special minimal atlas corresponding to the homeomorphism $\Vmap\circ\idV^{-1} \in \aHRzp$, i.e.\ $\UAtlas = \MinAtlas{\Vmap\circ\idV^{-1}}$;

\item\label{enum:spec_minimal_atlas:characterization}
For a $\Ck$-structure $\Ustruct$ on $\DLine$ the following conditions are equivalent:
\begin{enumerate}[label={\rm(\alph*)}, itemsep=1ex]
\item\label{enum:exist_spec_red_atl:exist} $\Ustruct$ has a special minimal $\Ck$-atlas;
\item\label{enum:exist_spec_red_atl:umap_diff}
there exists a minimal atlas $\UAtlas=\{ (\Uset,\Umap), (\Vset, \Vmap) \} \in \Ustruct$ such that $\Umap\colon\Uset=\bR \to \bR$ is a $\Ck$-diffeomorphism;
\item\label{enum:exist_spec_red_atl:any_umap_diff}
for every minimal atlas $\UAtlas=\{ (\Uset,\Umap), (\Vset, \Vmap) \} \in \Ustruct$ we have that $\Umap\colon\Uset=\bR \to \bR$ is a $\Ck$-diffeomorphism;
\end{enumerate}

\item\label{enum:spec_minimal_atlas:diff}
Let $\Ustruct$ be a $\Ck$-structure on $\DLine$, $\UAtlas=\{ (\Uset,\Umap), (\Vset, \Vmap) \} \in \Ustruct$ some minimal atlas, and $\Utrmap$ be the $0$-extension of its transition map $\Vmap\circ\Umap^{-1}$;
Then the $\Ck$-structures $\Ustruct$ and $\trStruct{\Utrmap}$ are diffeomorphic via a preserving orientation diffeomorphism $\dif$ fixing $\pta$ and $\ptb$, i.e.\ $\DiffLxOrPresxOFix[\Ustruct, \trStruct{\Utrmap}]\ne\varnothing$;
For instance, if $\Umap$ preserves orientation, then the following map $\dif\colon\DLine\to\DLine$ defined by
\[
    \restr{\dif}{\Uset} = \idU^{-1}\circ\Umap\colon\Uset\to\Uset
    \qquad \text{and} \qquad \dif(\ptb)=\ptb
\]
belongs to $\DiffLxOrPresxOFix[\Ustruct, \trStruct{\Utrmap}]$;

\item\label{enum:lm:spec_reduc_atlas:compose_with_diff}
Let $\gdif,\dif\in\DiffRz$.
Then $\trStruct{\dif} = \trStruct{\gdif}$ if and only if $\gdif\circ\dif^{-1}\in\DiffRz$.
In particular, there is a bijection between $\Ck$-structures on $\DLine$ admitting special minimal atlas and \term{right} cosets $\DiffRz\setminus \aHRzp$;

\item\label{enum:lm:spec_reduc_atlas:DPlab_DMinainvb}
For each $\gdif,\dif\in\aHRzp$ there are bijections:
\[
    \DiffLxOFix[\trStruct{\gdif},\trStruct{\dif}]
    \xrightarrow{~ \bijDL[\MinAtlas{\gdif},\MinAtlas{\dif}]{} ~}
    \DiffRz \ \cap \ \dif\,\DiffRz\,\gdif^{-1}
    \xleftarrow{~ \bijDL[\MinAtlas{\gdif^{-1}},\MinAtlas{\dif}]{} ~}
    \DiffLxOEx[\trStruct{\gdif^{-1}},\trStruct{\dif}];
\]

\item\label{enum:lm:spec_reduc_atlas:non_diff_structs}
If $\bdif\in\aHRzp$ but $\bdif\not\in\DiffRz$, then the $\Ck$-structures $\trStruct{\id_{\bR}}$ and $\trStruct{\bdif}$ are not $\Ck$-diffeomorphic, i.e.\ $\DiffLx[\trStruct{\id_{\bR}},\trStruct{\bdif}]{} = \varnothing$.
\end{enumerate}
\end{sublemma}
\begin{proof}
Statements~\ref{enum:spec_minimal_atlas:minimal_umap_id__spec} is trivial, while~\ref{enum:lm:spec_reduc_atlas:DPlab_DMinainvb} follows from Theorem~\ref{th:Dfix_Dex}.

\ref{enum:spec_minimal_atlas:characterization}
The implication \ref{enum:exist_spec_red_atl:any_umap_diff}$\Rightarrow$\ref{enum:exist_spec_red_atl:umap_diff} is also trivial.

\ref{enum:exist_spec_red_atl:umap_diff}$\Rightarrow$\ref{enum:exist_spec_red_atl:exist}
Let $\UAtlas=\{ (\Uset,\Umap), (\Vset, \Vmap) \} \in \Ustruct$ be a $\Ck$-atlas such that $\Umap\colon\Uset=\bR \to \bR$ is a $\Ck$-diffeomorphism.
Then $\{ (\Uset,\Umap^{-1}\circ\Umap = \idU), (\Vset, \Vmap) \}$ is a special minimal $\Ck$-atlas $\Ck$-compatible with $\UAtlas$, and therefore belonging to $\Vstruct$.

\ref{enum:exist_spec_red_atl:exist}$\Rightarrow$\ref{enum:exist_spec_red_atl:any_umap_diff}
Let $\MinAtlas{\adif}=\{ (\Uset,\idU), (\Vset, \adif\circ\idV) \} \in \Ustruct$ be a special minimal atlas, and $\UAtlas=\{ (\Uset,\Umap), (\Vset, \Vmap) \}$ be any other minimal atlas in $\Ustruct$.
Then they must be $\Ck$-compatible, i.e.\ $\Umap\circ\idU^{-1} = \Umap$ is a $\Ck$-diffeomorphism.

\ref{enum:spec_minimal_atlas:diff}
By Theorem~\ref{th:Dfix_Dex} we have a bijection $\DiffLxOFix[\Ustruct,\trStruct{\Utrmap}] \equiv \DiffRz \cap \Utrmap \DiffRz \Utrmap^{-1}$.
The latter intersection contains $\id_{\bR}$, and therefore it is non-empty, whence $\Ustruct$ and $\trStruct{\Utrmap}$ are $\Ck$-diffeomorphic via a diffeomorphism fixing $\pta$ and $\ptb$.

\ref{enum:lm:spec_reduc_atlas:compose_with_diff}
Consider the corresponding special minimal atlases
\begin{align*}
    \MinAtlas{\gdif}  &:= \{ (\Uset,\idU), (\Vset, \gdif\circ\idV) \}, &
    \MinAtlas{\dif}   &:= \{ (\Uset,\idU), (\Vset, \dif\circ\idV)  \}.
\end{align*}
Since they have the same first chart $(\Uset,\idU)$, the $\Ck$-structures $\MinAtlas{\gdif}$ and $\MinAtlas{\dif}$ coincide if and only if their second charts are $\Ck$-compatible.
The latter means that the corresponding transition map $(\dif\circ\idV) \circ (\gdif\circ\idV)^{-1} = \dif\circ\gdif^{-1}\in\DiffRz$.

\ref{enum:lm:spec_reduc_atlas:non_diff_structs}
By Theorem~\ref{th:Dfix_Dex} we have the following bijections:
\begin{equation}\label{equ:DABpl_a_idR}
\begin{aligned}
    \DiffLxOFix[\trStruct{\id_{\bR}},\trStruct{\bdif}]
        &
        \stackrel{\bijDL[\MinAtlas{\id_{\bR}},\MinAtlas{\bdif}]{}}{\equiv}
    \DiffRz \cap \bdif\DiffRz\id_{\bR}^{-1}
        =
    \DiffRz \cap \bdif\DiffRz
        = \\
        &=
    \DiffRz \cap \bdif\DiffRz\id_{\bR}
        \stackrel{(\bijDL[\MinAtlas{\id_{\bR}^{-1}},\MinAtlas{\bdif}]{})^{-1}}{\equiv}
    \DiffLxOEx[\trStruct{\id_{\bR}},\trStruct{\bdif}].
\end{aligned}
\end{equation}
Since $\bdif\not\in\DiffRz$, the adjacent classes $\DiffRz$ and $\bdif\DiffRz$ do not intersect each other, so all the sets in~\eqref{equ:DABpl_a_idR} are empty.
Hence,
\[
    \DiffLx[\trStruct{\id_{\bR}},\trStruct{\bdif}]{}
        \ = \
        \DiffLxOFix[\trStruct{\id_{\bR}},\trStruct{\bdif}]
        \ \sqcup \
        \DiffLxOEx[\trStruct{\id_{\bR}},\trStruct{\bdif}]
        \ = \
        \varnothing.
\]
Therefore, there are no diffeomorphisms between $\trStruct{\id_{\bR}}$ and $\trStruct{\bdif}$ (neiter fixing $\pta$ and $\ptb$ nor exchanging them).
\end{proof}

Statement~\ref{enum:spec_minimal_atlas:diff} of Lemma~\ref{lm:spec_minimal_atlas} shows that in order to classify $\Ck$-structures on $\DLine$ up to a $\Ck$-diffeomorphism, it suffices to consider only $\Ck$-structures having special minimal atlases.
Also, due to statement~\ref{enum:lm:spec_reduc_atlas:non_diff_structs}, $\DLine$ admits at least two non-diffeomorphic $\Ck$-structures.

Our final step is to show that in fact, there are uncountably many pair-wise non-diffeomorphic $\Ck$-structures on $\DLine$.

Recall that for each $\ax>0$ we have defined a homeomorphism $\stlin{\ax}$ by formula~\eqref{equ:stlin_b}, so $\stlin{\ax}(\px)=\px$ for $\px\leq0$, and $\stlin{\ax}(\px)=\ax\px$ for $\px\geq0$.
The following simple statement allows to classify $\Ck$-structures $\trStruct{\stlin{\ax}}$ on $\DLine$ corresponding to special minimal atlases $\MinAtlas{\stlin{\ax}}$ of $\stlin{\ax}$ for all $\ax>0$.

\begin{sublemma}\label{lm:l_ab_homeo}
Let $\ax,\bx>0$, $\func\in\DiffRz$, $\tau_1=\func'(0)$, and
\[ \qfunc := \stlin{\bx} \circ \func \circ \stlin{\ax} \in\aHRzp.\]
\begin{enumerate}[label={\rm(\alph*)}, leftmargin=*, itemsep=1ex]
\item\label{enum:lm:l_ab_homeo:tau_pos}
Suppose $\tau_1>0$;
Then $\qfunc$ is $\Cr{1}$ if and only if $\ax=\inv{\bx}$.

\item\label{enum:lm:l_ab_homeo:tau_neg}
Let $\tau_1<0$.
Then $\qfunc$ is $\Cr{1}$ if and only if $\ax=\bx$;

\item\label{enum:lm:l_ab_homeo:cn}
Suppose $\qfunc$ is $\Cr{1}$ and let $n\in\bN$ be such that $2\leq n\leq k$.
Assume also that $\ax,\bx\ne1$.
Then $\qfunc$ is $\Cr{n}$ if and only if $\func''(0)=\cdots=\func^{(n)}(0)=0$.
\end{enumerate}
\end{sublemma}
\begin{proof}
Since $\func$ is of class $\Ck$, we have that $\func(\px) = \tau_1\px +  \tau_2\px^2 + \cdots + \tau_n \px^n + \gfunc(\px)$, where $\tau_i = \tfrac{1}{i!} \func^{(i)}(0)$, $i=1,\ldots,n$, and $\gfunc\colon\bR\to\bR$ is a $\Ck$-function such that $\lim\limits_{\px\to0}\bigl(\gfunc(\px)/\px^{n}\bigr)=0$.
It follows that if $\tau_1>0$, then
\begin{equation}\label{equ:q_dfb_tau_pos}
    \qfunc(\px) =
    \begin{cases}
    \func(\px) =
        \tau_1\px +
        \tau_2\px^2 +
        \cdots
        \tau_n\px^n +
        \gfunc(\px), & \px\leq0, \\[1mm]
   \bx \func(\ax\px) =
         \bx\ax\tau_1\px +
         \bx\ax^2\tau_2\px^2 +
         \cdots
         \bx\ax^n\tau_n\px^n +
         \bx\gfunc(\ax\px), & \px\geq0.
    \end{cases}
\end{equation}
On the other hand, if $\tau_1<0$, then
\begin{equation}\label{equ:q_dfb_tau_neg}
    \qfunc(\px) =
    \begin{cases}
    \bx\func(\px) =
        \bx\tau_1\px +
        \bx\tau_2\px^2 +
        \cdots
        \bx\tau_n\px^n +
        \bx\gfunc(\px), & \px\leq0, \\[1mm]
    \func(\ax\px) =
         \ax\tau_1\px +
         \ax^2\tau_2\px^2 +
         \cdots
         \ax^n\tau_n\px^n +
         \gfunc(\ax\px), & \px\geq0.
    \end{cases}
\end{equation}

Thus if $\tau_1>0$, then due to~\eqref{equ:q_dfb_tau_pos}, $\qdif$ is $\Cr{1}$ iff $\bx\ax=1$.
Moreover, in that case $\qdif$ is $\Cr{n}$ iff $\ax^{i-1}\tau_i=\tau_i$ for $i=2,\ldots,n$.
Since $\ax \ne 0$, the latter is possible iff $\tau_i=0$ for $i=2\ldots,n$.
This proves~\ref{enum:lm:l_ab_homeo:tau_pos} and a part of~\ref{enum:lm:l_ab_homeo:cn} for $\tau_1>0$.

Similarly, if $\tau_1<0$, then due to~\eqref{equ:q_dfb_tau_neg}, $\qdif$ is $\Cr{1}$ iff $\bx=\ax$.
Moreover, again that case $\qdif$ is $\Cr{n}$ iff $\ax^{i-1}\tau_i=\tau_i$ for $i=2,\ldots,n$, which is again equivalent to the assumption that $\tau_i=0$ for $i=2\ldots,n$.
This proves~\ref{enum:lm:l_ab_homeo:tau_neg} and completes~\ref{enum:lm:l_ab_homeo:cn} for $\tau_1<0$.
\end{proof}

\begin{subcorollary}\label{cor:D_cap_wb_D_wa}
For all $\ax,\bx \ne 1$ we have that
\begin{align*}
    \DiffRz \cap \stlin{\bx} \DiffRz \stlin{\ax} &=
    \begin{cases}
        \varnothing,     & \ax \ne \bx, \inv{\bx}, \\[1mm]
        \JDiffRzOrPres,  & \ax\bx = 1,             \\[1mm]
        \JDiffRzOrRev,   & \ax = \bx.
    \end{cases}
\end{align*}
\end{subcorollary}
\begin{proof}
1) If $\ax\ne \bx,\inv{\bx}$, then by Lemma~\ref{lm:l_ab_homeo} for any $\func\in\DiffRz$ the homeomorphism $\qdif=\stlin{\bx} \circ\func\circ \stlin{\ax}$ is not even $\Cr{1}$.
In particular, $\DiffRz \cap \stlin{\bx} \DiffRz \stlin{\ax}=\varnothing$.

2) Suppose $\ax=\inv{\bx}$.
Again, let $\func\in\DiffRz$ and $\qdif=\stlin{\bx} \circ\func\circ \stlin{\ax}$.

If $\func\in\DiffRzOrRev$, then by~\ref{enum:lm:l_ab_homeo:tau_neg}, $\qdif$ is not $\Cr{1}$ since $\ax\ne\bx$.
On the other hand, if $\func\in\DiffRzOrPres$, then by~\ref{enum:lm:l_ab_homeo:tau_pos} and~\ref{enum:lm:l_ab_homeo:cn}, $\qdif$ is $\Ck$ iff $\func\in\JDiffRzOrPres$.
In that case, it also follows from~\eqref{equ:q_dfb_tau_pos} that $\qdif\in\JDiffRzOrPres$.
In other words,
\[
    \DiffRz \cap \stlin{\bx} \DiffRz \stlin{\ax}
    \,=\,
    \DiffRz \cap \stlin{\bx} \JDiffRzOrPres \stlin{\ax}
    \,\subset\,
    \JDiffRzOrPres.
\]

To prove the inverse inclusion take any $\bar{\qdif}\in\JDiffRzOrPres$, and let $\bar{\func} = \stlin{\inv{\bx}} \circ\bar{\qdif}\circ \stlin{\inv{\ax}}$.
Since $\inv{\ax} = 1/(\inv{\bx})$, the already proved part of case 2) applied to $\bar{\qdif}$ implies that $\bar{\func} \in \JDiffRzOrPres$.
Hence,
\[
    \bar{\qdif} = \stlin{\bx} \circ\bar{\func}\circ \stlin{\ax}\in \DiffRz \cap \stlin{\bx} \JDiffRzOrPres \stlin{\ax}.
\]

3) The proof that $\DiffRz \cap \stlin{\bx} \DiffRz \stlin{\ax} = \JDiffRzOrRev$ for $\ax=\bx\ne 1$ is similar.
\end{proof}

\section{Proof of Theorem~\ref{th:ck_structs_examples}}\label{th:proof:th:ck_structs_examples}
Denote by $\trStruct{\ax}$ the $\Ck$-structure on $\DLine$ corresponding to the special minimal atlas $\MinAtlas{\stlin{\ax}}$ of the homeomorphism $\stlin{\ax}$.
We need to describe the sets of diffeomorphisms between $\Ck$-structures $\trStruct{\ax}$ and $\trStruct{\bx}$ for all $\ax,\bx>0$.
Let
\begin{align*}
    &\MinAtlas{\stlin{\ax}} = \{(\Na, \idU), (\Nb, \stlin{\ax}\circ\idV) \}, &
    &\MinAtlas{\stlin{\bx}} = \{(\Na, \idU), (\Nb, \stlin{\bx}\circ\idV) \},
\end{align*}
be special minimal atlases of $\stlin{\ax}$ and $\stlin{\bx}$ respectively.
Then by Theorem~\ref{th:Dfix_Dex} we have the following bijections:
\begin{align*}
    \bijDLxOFix\colon&\DiffLxOFix[\trStruct{\ax},\trStruct{\bx}] \to
        \DiffRz \cap \stlin{\bx} \DiffRz \stlin{\ax}^{-1},
        & \bijDLxOFix(\dif) &=  \stlin{\bx}\circ\idV\circ\dif\circ\idV^{-1}\circ\stlin{\ax}^{-1}, \\
    \bijDLxOEx\colon&\DiffLxOFix[\trStruct{\ax},\trStruct{\bx}] \to
        \DiffRz \cap \stlin{\bx} \DiffRz \stlin{\ax},
        & \bijDLxOEx(\dif) &= \idU\circ\dif\circ\idV^{-1}\circ\stlin{\ax}^{-1}.
\end{align*}
given by the lower arrows of the corresponding diagrams:
\begin{align*}
    &\xymatrix@C=5em{
        \Nb \ar[r]^-{\dif} \ar[d]_-{\stlin{\ax}\circ\idV} & \Nb \ar[d]^-{\stlin{\bx}\circ\idV} \\
        \bR \ar[r]^-{\bijDLxOFix(\dif)} & \bR
    }
    &&
    \xymatrix@C=5em{
        \Nb \ar[r]^-{\dif} \ar[d]_-{\stlin{\ax}\circ\idV} & \Na \ar[d]^-{\idU} \\
        \bR \ar[r]^-{\bijDLxOEx(\dif)} & \bR
    }
\end{align*}
and thus associating to each $\dif\in\DiffLx[\trStruct{\ax},\trStruct{\bx}]{}$ its local presentation from chart $\Nb$.
Moreover, Corollary~\ref{cor:D_cap_wb_D_wa} gives an explicit description of the images of $\bijDLxOFix$ and $\bijDLxOEx$.

\ref{enum:th:ck_structs_examples:a_ne_b_1b}
If $\ax\not=\bx,\inv{\bx}$, then by Corollary~\ref{cor:D_cap_wb_D_wa} the intersections $\DiffRz \cap \stlin{\bx} \DiffRz \stlin{\ax}^{\pm1}$ are empty, whence $\DiffLx[\trStruct{\ax},\trStruct{\bx}]{}=\varnothing$ as well.

\ref{enum:th:ck_structs_examples:a_1}
Let $\ax=1$, so $\stlin{\ax} = \id_{\bR}$.
Then
\[
    \stlin{1} \DiffRz \stlin{1}^{-1} = \DiffRz \cap \stlin{1} \DiffRz \stlin{1} = \DiffRz,
\]
whence we above bijections reduce to the following ones:
\begin{align*}
   &\bijDLxOFix\colon \DiffLxOFix[\trStruct{1}] \to \DiffRz, &
   &\bijDLxOEx \colon \DiffLxOEx[\trStruct{1}] \to \DiffRz.
\end{align*}

Let $\invol\colon\DLine\to\DLine$ be the involution exchanging $\pta$ and $\ptb$ and fixed on $\Rzp$.
One easily checks $\invol$ is a self-diffeomorphism of $\DiffLxOEx[\trStruct{1}]$.
Namely, the local presentations of $\invol$ with respect to the special minimal atlas $\MinAtlas{\stlin{1}} = \{(\Na, \idU), (\Nb, \stlin{1}\circ\idV = \idV) \}$ are given by the following commutative diagrams in which both lower arrows are $\id_{\bR}$:
\begin{align*}
    &\xymatrix@R=4ex@C=12ex{
        \Na \ar[r]^-{\invol}  \ar[d]_-{\idU} & \Nb \ar[d]^-{\idV} \\
        \bR \ar[r]^-{\id_{\bR}} & \bR
    }
    &
    &\xymatrix@R=4ex@C=12ex{
        \Nb \ar[r]^-{\invol}  \ar[d]_-{\idV} & \Na \ar[d]^-{\idU} \\
        \bR \ar[r]^-{\id_{\bR}} & \bR
    }
\end{align*}
In particular, $\bijDLxOEx(\invol) = \id_{\bR}$.

Moreover, as noted in Lemma~\ref{lm:HomeoL}, $\invol$ commutes with all homeomorphisms of $\DLine$, and in particular with $\DiffLxOFix[\trStruct{1}]$.
Hence the map $\eta\colon\DiffRz \times \bZ_{2} \to \DiffLx{}$, $\eta(\dif,\delta) \mapsto \invol^{\delta}\circ\dif$, is an isomorphism of groups.
It is aslo evident, that is the identities~\eqref{equ:eta_a=1} hold.

\ref{enum:th:ck_structs_examples:a_ne_1}
Suppose $\ax \ne 1$.
Note also that $\stlin{\ax}^{-1}=\stlin{\inv{\ax}}$.
Then by Theorem~\ref{th:Dfix_Dex} and Corollary~\ref{cor:D_cap_wb_D_wa} the above bijections reduce to the following ones
\begin{align*}
    \bijDLxOFix &\colon \DiffLxOFix[\trStruct{\ax}] \to \DiffRz \cap \stlin{\ax} \DiffRz \stlin{\inv{\ax}} = \JDiffRzOrPres,
    \\
    \bijDLxOEx  &\colon \DiffLxOEx[\trStruct{\ax}] \to \DiffRz \cap \stlin{\ax} \DiffRz \stlin{\ax} = \JDiffRzOrRev,
\end{align*}
which send preserving (resp.\ reversing) orientation diffeomorphisms of $\DLine$ to diffeomorphisms of $\bR$ of the same type.
Hence $\DiffLxOrRevxOFix[\trStruct{\ax}]=\DiffLxOrPresxOEx[\trStruct{\ax}] = \varnothing$, and
\[ \DiffLx[\trStruct{\ax}]{}=\DiffLxOrPresxOFix[\trStruct{\ax}]{}\sqcup\DiffLxOrRevxOEx[\trStruct{\ax}]{}. \]

Finally, consider the following homeomorphism $\pdif\colon\DLine\to\DLine$:
\[
\pdif(\px) =
    \begin{cases}
        -\px/\sqrt{\ax}, & \px\in(-\infty;0),\\
        \ptb,     & \px=\pta,\\
        \pta,     & \px=\ptb, \\
        -\px\sqrt{\ax}, & \px\in(0;\infty).
    \end{cases}
\]
It exchanges $\pta$ and $\ptb$ and one easily checks that $\pdif$ has order $2$.
We claim that $\pdif$ is the desired $\Ck$-diffeomorphism of $\trStruct{\ax}$ belonging to $\DiffLxOrRevxOEx[\trStruct{\ax}]$.

To check that, let us will compute local presentations of $\pdif$ in each of the charts of the special minimal atlas $\MinAtlas{\stlin{\ax}} = \{(\Na, \idU), (\Nb, \stlin{\ax}\circ\idV) \}$ of $\stlin{\ax}$.
Note that $\pdif$ exchanges the connected components of $\DLine\setminus\{\pta,\ptb\} = \Rzp$.
Then the local presentation of $\pdif$ with respect to the charts $(\Na, \idU)$ and $(\Nb, \stlin{\ax}\circ\idV)$ on the components $(-\infty;0)$ and $(0;\infty)$ are given by the following two diagrams:
\begin{align*}
    &\xymatrix@R=4ex@C=12ex{
        (-\infty;0) \ar[r]^-{\pdif(\px)=-\px/\sqrt{\ax}}  \ar[d]_-{\idU} & (0;+\infty) \ar[d]^-{\stlin{\ax}\circ\idV(\px) = \ax\px} \\
        (-\infty;0) \ar[r]^-{\px\mapsto -\sqrt{\ax}\px} & (0;+\infty)
    }
    &
    &\xymatrix@R=4ex@C=12ex{
        (0;\infty) \ar[r]^-{\pdif(\px)=-\sqrt{\ax}\px}  \ar[d]_-{\idU} & (-\infty;0) \ar[d]^-{\stlin{\ax}\circ\idV(\px) = \px} \\
        (0;\infty) \ar[r]^-{\px\mapsto -\sqrt{\ax}\px} & (-\infty;0)
    }
\end{align*}
The formulas for the lower arrows are the same, and thus the local presentation of $\pdif$ with respect to the charts $(\Na, \idU)$ and $(\Nb, \stlin{\ax}\circ\idV)$ is the map $\px \mapsto -\sqrt{\ax}\px$.

Similarly, the local presentation of $\pdif$ with respect to the charts $(\Nb, \stlin{\ax}\circ\idV)$ and $(\Na, \idU)$ on the components $(-\infty;0)$ and $(0;\infty)$ are given by:
\begin{align*}
    &\xymatrix@R=4ex@C=12ex{
        (-\infty;0) \ar[r]^-{\pdif(\px)=-\px/\sqrt{\ax}} \ar[d]_-{\stlin{\ax}\circ\idV(\px) = \px}
         & (0;+\infty)  \ar[d]^-{\idU} \\
        (-\infty;0) \ar[r]^-{\px\mapsto -\px/\sqrt{\ax}} & (0;+\infty)
    }
    &
    &\xymatrix@R=4ex@C=12ex{
        (0;\infty) \ar[r]^-{\pdif(\px)=-\sqrt{\ax}\px}  \ar[d]_-{\stlin{\ax}\circ\idV(\px) = \ax\px}
        & (-\infty;0) \ar[d]^-{\idU} \\
        (0;\infty) \ar[r]^-{\px\mapsto -\px/\sqrt{\ax}} & (-\infty;0)
    }
\end{align*}
Hence, the local presentation of $\pdif$ with respect to the charts $(\Nb, \stlin{\ax}\circ\idV)$ and $(\Na, \idU)$ is the map $\px \mapsto -\px/\sqrt{\ax}$.

Thus $\pdif$ is $\Ck$ in both charts of the atlas $\MinAtlas{\stlin{\ax}}$, and therefore it is a self-diffeomorphism of $\trStruct{\ax}$.

\ref{enum:th:ck_structs_examples:a_ainv}
Again if $\ax\not=1$, then by Theorem~\ref{th:Dfix_Dex} and Corollary~\ref{cor:D_cap_wb_D_wa} the above bijections reduce to the following ones
\begin{align*}
    \bijDLxOFix &\colon \DiffLxOFix[\trStruct{\ax},\trStruct{\inv{\ax}}] \to \DiffRz \cap \stlin{\inv{\ax}} \DiffRz \stlin{\ax}^{-1} = \JDiffRzOrRev,
    \\
    \bijDLxOEx  &\colon \DiffLxOEx[\trStruct{\ax},\trStruct{\inv{\ax}}] \to \DiffRz \cap \stlin{\inv{\ax}} \DiffRz \stlin{\ax} = \JDiffRzOrPres,
\end{align*}
which send preserving (resp.\ reversing) orientation diffeomorphisms of $\DLine$ to diffeomorphisms of $\bR$ of the same type.

Hence $\DiffLxOrPresxOFix[\trStruct{\ax},\trStruct{\inv{\ax}}] = \DiffLxOrRevxOEx[\trStruct{\ax},\trStruct{\inv{\ax}}] = \varnothing$, and
\[
    \DiffLx[\trStruct{\ax},\trStruct{\inv{\ax}}]{}
    =
    \DiffLxOrRevxOFix[\trStruct{\ax},\trStruct{\inv{\ax}}]{}
    \sqcup
    \DiffLxOrPresxOEx[\trStruct{\ax},\trStruct{\inv{\ax}}]{}.
\]

The latter statement on bijections of those sets with $\DiffLxOrPresxOFix[\trStruct{\ax}]$ follows from Lemma~\ref{lm:DiffM_DiffN}.
This completes Theorem~\ref{th:ck_structs_examples}.

\subsection*{Acknowledgement}
This work was supported by grants from the Simons Foundation (1030291, 1290607, S.I.M. and L.M.V.).